\documentclass[11pt,letterpaper]{amsart}
\makeatletter

\usepackage[utf8]{inputenc}
\usepackage[T1]{fontenc}
\usepackage{lmodern}
\usepackage[english]{babel}
\usepackage[stretch=50,shrink=50]{microtype}
\usepackage{amsmath,amssymb,amsfonts,amsthm,esint}
\usepackage{mathtools,accents}
\usepackage{mathrsfs}
\usepackage{aliascnt}
\usepackage{bm}
\usepackage[citecolor=blue,colorlinks]{hyperref}
\usepackage{enumitem}
\usepackage{graphicx}

\makeatletter
\def\newaliasedtheorem#1[#2]#3{
	\newaliascnt{#1@alt}{#2}
	\newtheorem{#1}[#1@alt]{#3}
	\expandafter\newcommand\csname #1@altname\endcsname{#3}
}
\makeatother

\numberwithin{equation}{section}

\newtheoremstyle{slanted}{\topsep}{\topsep}{\slshape}{}{\bfseries}{.}{.5em}{}

\theoremstyle{plain}
\newtheorem{theorem}{Theorem}[section]
\newaliasedtheorem{proposition}[theorem]{Proposition}
\newaliasedtheorem{lemma}[theorem]{Lemma}
\newaliasedtheorem{corollary}[theorem]{Corollary}
\newaliasedtheorem{counterexample}[theorem]{Counterexample}

\theoremstyle{definition}
\newaliasedtheorem{definition}[theorem]{Definition}
\newaliasedtheorem{question}[theorem]{Question}
\newaliasedtheorem{openquestion}[theorem]{Open Question}
\newaliasedtheorem{conjecture}[theorem]{Conjecture}

\theoremstyle{remark}
\newaliasedtheorem{remark}[theorem]{Remark}
\newaliasedtheorem{example}[theorem]{Example}

\newcommand{\N}{\mathbb{N}}

\newcommand{\setR}{\mathbb{R}}
\newcommand{\R}{\mathbb{R}}
\newcommand{\Z}{\mathbb{Z}}

\newcommand{\eps}{\varepsilon}
\renewcommand{\epsilon}{\varepsilon}
\renewcommand{\phi}{\varphi}

\renewcommand{\tilde}{\widetilde}

\newcommand{\di}{\mathop{}\!\mathrm{d}}

\newcommand{\diam}{{\rm diam}}

\DeclareMathOperator{\Ric}{Ric}

\newcommand{\haus}{\mathscr{H}}

\newcommand{\dist}{\mathsf{d}}

\newcommand{\meas}{\mathfrak{m}}

\DeclareMathOperator{\CD}{CD}
\DeclareMathOperator{\RCD}{RCD}

\newfont{\tmpf}{cmsy10 scaled 2500}

\def\XXint#1#2#3{{\setbox0=\hbox{$#1{#2#3}{\int}$ }
		\vcenter{\hbox{$#2#3$ }}\kern-.6\wd0}}

\begin{document}

\title[]{The large scale structure of complete $4$-manifolds with nonnegative Ricci curvature and Euclidean volume growth}

\author[]{Daniele Semola}

\address{\parbox{\linewidth}{University of Vienna, Faculty of Mathematics\\
	 Oskar-Morgenstern-Platz 1,
    1090 Wien -- Austria\\[-4pt]\phantom{a}}}
\email{daniele.semola@univie.ac.at}

\begin{abstract}
We survey the implications of our joint work with E. Bru\`e and A. Pigati \cite{BruePigatiSemola} on the structure of blow-downs for a smooth, complete, Riemannian $4$-manifold with nonnegative Ricci curvature and Euclidean volume growth. Very imprecisely, any such manifold looks like a cone over a spherical space form at infinity. We present some open questions and discuss possible future directions along the way. 
\end{abstract}

\maketitle


\section{Introduction}\label{sec:Intro}

In the recent \cite{BruePigatiSemola}, joint work with E. Bru\`e  and A. Pigati, we obtained some topological regularity and stability results for noncollapsed Ricci limit spaces and $\RCD$ metric measure spaces, i.e., Riemannian metric measure spaces with Ricci curvature bounded from below in the synthetic sense (see \cite{AmbrosioICM,SturmECM,GigliSurvey} for some background). This survey aims to review the implications of \cite{BruePigatiSemola} on the large-scale structure of a smooth, complete Riemannian $4$-manifold $(M^4,g)$ with nonnegative Ricci curvature and Euclidean volume growth. All Riemannian manifolds will be assumed to be smooth and complete in this note unless otherwise stated. 

Some familiarity with the basics of the theory of Ricci limit spaces, as covered in J. Cheeger's notes \cite{CheegerFermi}, might help the reader. Most arguments will be only sketched, as we try to balance precision with informality.
\medskip

We recall that $(M^4,g)$ with $\Ric\ge 0$ is said to have \textbf{Euclidean volume growth} provided that there exists $c>0$ such that for some (and hence for every) $p\in M^4$ it holds
\begin{equation}\label{eq:EVG}
\frac{\mathrm{vol}(B_r(p))}{r^4}\ge c\quad \text{for all $r>0$}\, .
\end{equation}
By the Bishop-Gromov volume monotonicity, \eqref{eq:EVG} is equivalent to
\begin{equation}
\lim_{r\to \infty}\frac{\mathrm{vol}(B_r(p))}{r^4}\in [c,\omega_4]\, ,
\end{equation}
where we denote by $\omega_4$ the volume of the unit ball $B_1(0)\subset \R^4$. 

\medskip
We shall see that for every such $(M^4,g)$, there exists a spherical space form $S^3/\Gamma$ such that every \textbf{blow-down} of $(M^4,g)$ is a cone with \textbf{cross-section} $\approx S^3/\Gamma$. The symbol ``$\approx$'' will stand for ``homeomorphic'' hereafter. 

We find it helpful to mention a result obtained by J. Cheeger and A. Naber for Ricci flat $4$-manifolds satisfying \eqref{eq:EVG} around ten years ago in \cite{CheegerNaber15}. This should provide some intuition before we give a precise statement. 

\begin{theorem}\label{thm:ChNaALE}{\cite[Corollary 8.85]{CheegerNaber15}} Let $(M^4,g)$ be a Ricci-flat $4$-manifold satisfying \eqref{eq:EVG}. 
There exists a finite group $\Gamma<\mathrm{O}(4)$ acting freely on $S^3$ such that $(M^4,r^{-2}g,p)\to (\R^4/\Gamma,g_{\rm{eucl}},o)$ as $r\to \infty$ in the pointed Gromov-Hausdorff sense and in $C^{\infty}_{{\rm loc}}$ away from $p$ and $o$. 
\end{theorem}

For later purposes, it is convenient to view $(\R^4/\Gamma,g_{\rm{eucl}},o)$ as the cone over $S^3/\Gamma$, where $S^3/\Gamma$ is endowed with the round metric with constant curvature $\equiv 1$. Our goal below is to illustrate which aspects of \autoref{thm:ChNaALE} continue to hold and which ones fail when the assumption $\Ric\equiv 0$ is weakened to $\Ric\ge 0$.
\medskip

Let $(M^4,g)$ have $\Ric\ge 0$ and satisfy \eqref{eq:EVG}. By M. Gromov's precompactness theorem, for any sequence $r_i\to \infty$, up to the extraction of a subsequence that we do not relabel, $(M^4,r_i^{-2}g,p)\to (Y,\dist_Y,q)$ in the pointed Gromov-Hausdorff sense (from now on abbreviated as ${\rm pGH}$), where $(Y,\dist_Y,q)$ is a complete and pointed metric space. Any such metric space is called a blow-down of $(M^4,g)$. Note that neither the dimension nor the Euclidean volume growth \eqref{eq:EVG} play any role for the moment.\\ 
Without further assumptions, the metric structure of blow-downs is not fully understood. On the other hand, if the manifold has Euclidean volume growth, Cheeger and T.-H. Colding proved in \cite[Theorem 7.6]{CheegerColding96} that every blow-down is a \textbf{metric cone}. More precisely, there exists a compact metric space $(Z,\dist_Z)$ (the cross-section of the cone) with $\rm{diam}(Z)\le \pi$ such that 
\begin{equation}
Y=[0,+\infty)\times Z/_{\{0\}\times Z}\, ,
\end{equation} 
and for every $(r_1,z_1),(r_2,z_2)\in Y$ it holds
\begin{equation}\label{eq:dC}
\dist_Y^2\left((r_1,z_1),(r_2,z_2)\right)=r_1^2+r_2^2-2r_1r_2\cos(\dist_Z(z_1,z_2))\, .
\end{equation}

\begin{remark}
An alternative proof can be obtained by exploiting the volume convergence \cite[Theorem 5.4]{CheegerColding97I}, the stability of the $\RCD(0,n)$ condition under pointed measured Gromov-Hausdorff convergence \cite{GigliMondinoSavare}, and the volume-cone implies metric cone for $\RCD$ spaces proved by G. De Philippis and N. Gigli in \cite[Theorem 1.1]{DePhilippisGigli}.
\end{remark}

\begin{remark}
Under the simplifying assumption that the cross-section $(Z,\dist_Z)$ is a smooth Riemannian manifold with metric $g_Z$, the cone distance in \eqref{eq:dC} is the distance on the completion of the smooth Riemannian metric $dr^2+r^2g_Z$ on $(0,+\infty)\times Z$.
\end{remark}

In \cite{Perelmannonunique}, G. Perelman constructed a manifold $(M^4,g)$ with $\Ric\ge0$ and Euclidean volume growth \eqref{eq:EVG} whose blow-down is not unique. This is the first fundamental difference with the Ricci flat case. The cross-sections of all blow-downs in Perelman's example are Berger spheres, i.e., smooth $S^3$'s endowed with a left-invariant Riemannian metric where one of the three directions in the Lie algebra is shrunk with respect to the others. \\ 
Let 
\begin{equation*}
\mathcal{C}_{\infty}:=\{(Z,\dist_Z)\, :\, (Z,\dist_Z) \, \text{is the cross-section of a blow-down of $(M^4,g)$} \}\, .
\end{equation*} 
Then $\mathcal{C}_{\infty}$ is compact and connected with respect to the Gromov-Hausdorff topology. This is a corollary of Gromov's compactness theorem, see for instance \cite[Theorem 4.2]{CheegerJiangNaber} for the details. Moreover, by volume convergence \cite[Theorem 5.4]{CheegerColding97I}, the $3$-dimensional Hausdorff measure $\mathcal{H}^3$ is constant (and finite) on $\mathcal{C}_{\infty}$. 
\medskip

As will be clear later, the smoothness of all the elements of $\mathcal{C}_{\infty}$ is generally not guaranteed.\footnote{If $(M^4,g)$ has quadratic Riemann curvature decay, as in Perelman's example, any $(Z,\dist_Z)$ is a $C^{1,\alpha}$ Riemannian manifold, for every $\alpha<1$.} However, under this extra condition, each blow-down $C(Z)$ of $(M^4,g)$ has $\Ric\ge 0$ in the smooth part. This is quite reasonable to expect but completely non-trivial to check.
An elementary computation shows this holds if and only if $\Ric_Z\ge 2$. By R. Hamilton's \cite[1.1 Main Theorem]{Hamilton}, any cross-section must be diffeomorphic to a spherical space form. Moreover, by Cheeger and Colding's stability \cite[Theorem A.1.3]{CheegerColding97I}, the diffeomorphism type is constant on $\mathcal{C}_{\infty}$. This is a completely non-trivial statement, as $\mathcal{C}_{\infty}$ is only known to be connected with respect to the Gromov-Hausdorff topology. 

\medskip

With the help of the $\RCD$ theory, we can justify the first assertions above, althoug only in a weak sense, without any smoothness assumption. Indeed, any smooth $n$-dimensional Riemannian manifold with $\Ric\ge 0$ is an $\RCD(0,n)$ metric measure space when endowed with its volume measure. Along any ${\rm pGH}$ converging sequence $(M^4,r_i^{-2}g,p)\to (C(Z),\dist_{C(Z)},o)$, the volume measures converge to the Hausdorff measure of the limit cone $\mathcal{H}^4$ by \cite[Theorem 5.4]{CheegerColding97I}. Moreover, it is not difficult to check that $\mathcal{H}^4=r^3dr\mathcal{H}^3_Z$ on $C(Z)$. By the stability of the $\RCD(0,n)$ condition under pointed measured Gromov-Hausdorff convergence, $(C(Z),\dist_{C(Z)},\mathcal{H}^4)$ is an $\RCD(0,4)$ space. By C. Ketterer's \cite[Theorem 1.1]{Ketterer15}, $(Z,\dist_Z,\mathcal{H}^3)$ is an $\RCD(2,3)$ space. 

\begin{remark}
Strictly speaking, Ketterer's theorem implies that $(Z,\dist_Z,\mathcal{H}^3)$ is an $\RCD^*(2,3)$ space. To infer that it is an $\RCD(2,3)$ space, one must rely on F. Cavalletti and E. Milman's \cite{CavallettiMilman}. The distinction between $\RCD$ and $\RCD^*$ spaces, which was present in most of the early papers on the theory, does not play any role for the purpose of this note, and the non-expert reader is encouraged to ignore it.
\end{remark}

As we shall clarify below, a straightforward combination of the main results obtained jointly with Bru\`e and Pigati in \cite{BruePigatiSemola} (see in particular the comment right after \cite[Theorem 1.1]{BruePigatiSemola}) yields:

\begin{theorem}\label{thm:mainIntro}
Let $(M^4,g)$ be smooth, complete, with $\Ric\ge 0$ and satisfying \eqref{eq:EVG}. There exists a finite group $\Gamma<\mathrm{O}(4)$ acting freely on $S^3$ such that for every cross-section of some blow-down $(Z,\dist_Z)\in \mathcal{C}_{\infty}$, $(Z,\dist_Z,\mathcal{H}^3)$ is an $\RCD(2,3)$ space with $Z\approx S^3/\Gamma$. 
\end{theorem}

There are three upshots for \autoref{thm:mainIntro}:
\begin{itemize}
\item[i)] The cross-section of every blow-down is a topological manifold;
\item[ii)] The possible topologies of the cross-sections are restricted;
\item[iii)] For a fixed $(M^4,g)$ the homeomorphism type is unique on $\mathcal{C}_{\infty}$. 
\end{itemize}
We will discuss in Section \ref{sec:sharp} in which sense \autoref{thm:mainIntro} is sharp. The focus of Section \ref{sec:previous} will be on some related previous developments. In Section \ref{sec:proof}, we will explain how to obtain \autoref{thm:mainIntro} from the main results in \cite{BruePigatiSemola} and outline the main ideas that enter into the proofs. In Section \ref{sec:OpenQuest} we will discuss some related open questions and possible future directions.
\medskip

\textbf{Acknowledgements.} I am grateful to Elia Bru\`e and Alessandro Pigati for the fruitful collaboration resulting in \cite{BruePigatiSemola}. I wish to thank Shouhei Honda, Christian Lange, Lucas Lavoyer, Chao Li, Alexander Lytchak, Tristan Ozuch, Marco Pozzetta, Dušan Repovš, Stephan Stadler, Burkhard Wilking, and Shengxuan Zhou for some interesting conversations on the topics of \cite{BruePigatiSemola}, for pointing out some useful references, and for helpful feedback on a preliminary version of this note.\\
This survey is an extended version of my two lectures on the occasion of the \emph{School and Conference on Metric Measure Spaces, Ricci Curvature, and Optimal Transport}, at Villa Monastero, Lake Como, in September 2024. I am grateful to the organizers for putting together such a nice event and the invitation to speak. 

\section{Sharpness of \autoref{thm:mainIntro}}\label{sec:sharp}
The goal of this section is to discuss in which sense \autoref{thm:mainIntro} is sharp and in which sense it might not.

\subsection{Realizing all possible topologies}\label{subsec:alltop}
In the recent \cite{Zhoucones}, S. Zhou constructed examples of $(M^4,g)$ with $\Ric\ge 0$ and Euclidean volume growth asymptotic to $C(S^3_{\delta}/\Gamma)$ for every finite $\Gamma<\mathrm{O}(4)$ acting freely on $S^3$. Here, $0<\delta=\delta(\Gamma)\le 1$ denotes the radius of $S^3$, which we assume to be endowed with a round metric with constant curvature. This result shows that every admissible topology for the cross-section of some blow-down according to \autoref{thm:mainIntro} can indeed arise. 

Several of the spherical space forms are known to appear as cross-sections of blow-downs in the Ricci flat context considered in \autoref{thm:ChNaALE}. For $\Gamma<\mathrm{SU}(2)$, this is due to P. Kronheimer in \cite{Kronheimer}, after some earlier important contributions by N. Hitchin, and G. Gibbons and S. Hawking. More recently, I. Suvaina \cite{Suvaina} and E. P. Wright \cite{Wright} added some cyclic subgroups of $\mathrm{O}(4)$ to the list by considering some quotients of Kronheimer's examples. In the Ricci flat case, it is an open question whether all the spherical space forms can arise as cross-sections of blow-downs. On the other hand, M. de Borbon and C. Spotti have shown in \cite{deBorbonSpotti} that all spherical space forms can arise in the context of Kähler Ricci-flat metrics with cone singularities. It seems conceivable to the author that the examples constructed in \cite{deBorbonSpotti} are $\RCD(0,4)$ spaces.
\medskip

The starting point in \cite{Zhoucones} is the observation that any finite subgroup of $\mathrm{O}(4)$ acting freely on $S^3$ is conjugate to a subgroup of $\mathrm{U}(2)$ in $\mathrm{O}(4)$, see e.g. \cite[Chapter 4.4]{Thurstonman}. It is then possible to construct examples asymptotic to $C(S^3_{\delta}/\Gamma)$ so that the underlying manifold $M^4$ is diffeomorphic to the minimal resolution of $\mathbb{C}^2/\Gamma$, see \cite[Section 2]{Zhoucones} for the relevant background. Moreover, for any given $\Gamma$, there exist infinitely many non-homotopically equivalent $(M^4,g)$ such that the cross-section of the blow-down is homeomorphic to $S^3/\Gamma$, see the remark after the proof of \cite[Theorem 1.1]{Zhoucones}.

\subsection{Realizing all possible metrics}\label{sec:Realization4d}
Recall that the space of Riemannian metrics $g$ with $\Ric\ge 2$ modulo isometries on a given $3$-manifold is path-connected with respect to the topology induced by Cheeger-Gromov convergence, thanks to \cite{Hamilton}. One can combine this statement with the methods developed by Colding and Naber in \cite{ColdingNabercones}, see for instance \cite[Section 4]{Zhoucones} for a similar argument, to prove the following: if $g$ is a smooth Riemannian metric with $\Ric\ge 2$ on a spherical space form $S^3/\Gamma$, then $(S^3/\Gamma,g)$ can be the cross-section of a blow-down of $(M^4,g)$ as in the assumptions of \autoref{thm:mainIntro}, up to possibly shrinking the metric. By pushing this observation further and relying on M. Simon's smoothing results from \cite{Simon14}, one gets:

\begin{proposition}\label{prop:Riclim}
Let $(S^3/\Gamma)$ be a spherical space form and let $g_i$ be smooth Riemannian metrics on $(S^3/\Gamma)$ such that $\Ric_i\ge 2$ and $\mathrm{vol}_i(S^3/\Gamma)>v>0$  for each $i\in \mathbb{N}$. Assume that 
\begin{equation}
\left(S^3/\Gamma,\dist_{g_i}\right)\xrightarrow{\mathrm{GH}} (S^3/\Gamma,\dist)\, ,\quad \text{as $i\to \infty$} \, .
\end{equation}
Then there exist $\delta\in (0,1]$ and a smooth $(M^4,g)$ with $\Ric\ge 0$ and Euclidean volume growth with a blow-down isometric to $C((S^3/\Gamma,\delta \dist))$. 
\end{proposition}

We can rephrase \autoref{prop:Riclim}, by saying that, up to possibly scaling the distances, it holds
\begin{equation}\label{eq:Riccross}
\overline{\left\{(N^3,g)\,:\, \Ric_g\ge 2\right\}}\subseteq \bigcup\mathcal{C}_{\infty}(M^4,g)\, .
\end{equation}
Above, the union ranges among all $(M^4,g)$ with $\Ric\ge 0$ and Euclidean volume growth and the closure at the left-hand side is taken with respect to the Gromov-Hausdorff topology. On the other hand, by \autoref{thm:mainIntro},
\begin{equation}\label{eq:crossRCD}
\bigcup\mathcal{C}_{\infty}(M^4,g)\subseteq \left\{(Z,\dist_Z)\, :\, (Z,\dist_Z,\mathcal{H}^3) \, \text{is an $\RCD(2,3)$ $3$-manifold}\right\}\, .
\end{equation}
In principle, both the inclusions in \eqref{eq:Riccross} and \eqref{eq:crossRCD} might be strict. We will come back to this point in Section \ref{sec:OpenQuest}. We also note that in general it is not clear whether the choice $\delta=1$ is allowed in \autoref{prop:Riclim}.

\subsection{Failure of topological regularity and uniqueness for $n>4$}
Both the topological regularity (i) and the topological uniqueness (iii) for sections of blow-downs might fail in dimensions $n>4$. 
\medskip

To illustrate the failure of topological regularity, we consider the Eguchi-Hanson metric $g_{\mathrm{EH}}$ on $T^*S^2$, i.e., the cotangent bundle of $S^2$. This is a Ricci flat metric with Euclidean volume growth, see \cite{EguchiHanson}. Its blow-down is $C(\mathbb{RP}^3)$, where $\mathbb{RP}^3$ is endowed with the round metric with constant curvature $1$. We let $g:=dr^2+g_{\mathrm{EH}}$ be the product metric on $\R\times T^*S^2$. It is straightforward to check that $g$ is Ricci flat and has Euclidean volume growth. The blow-down of $(\R\times T^*S^2,g)$ is (unique and) isometric to $\R\times C(\mathbb{RP}^3)=C(S^4/(\mathbb{Z}/2\mathbb{Z}))$. Here $\mathbb{Z}/2\mathbb{Z}$ acts as an involution of $S^4$ with two fixed points. The cross-section $S^4/(\mathbb{Z}/2\mathbb{Z})$ is an orbifold, but it is not a topological $4$-manifold. 
\medskip

The first counterexamples to the topological uniqueness of cross-section for $n>4$ were constructed by Colding and Naber in \cite[Theorem 1.3]{ColdingNabercones}. The starting point is the following:
\begin{lemma}{\cite[Section 4]{ColdingNabercones}}\label{lemma:instab4d}
There exists a smooth family $(g_t)_{t\in(0,1]}$ of smooth Riemannian metrics on $\mathbb{CP}^2\#\overline{\mathbb{CP}^2}$ such that 
\begin{itemize}
\item[i)] $\Ric_t\ge 3$ for each $t\in (0,1]$;
\item[ii)] $\mathrm{vol}_t(\mathbb{CP}^2\#\overline{\mathbb{CP}^2})>v>0$ for each $t\in (0,1]$;
\item[iii)] $(\mathbb{CP}^2\#\overline{\mathbb{CP}^2},\dist_{g_t})\xrightarrow{\mathrm{GH}}(S^4,\dist)$, as $t\to 0$.
\end{itemize}
\end{lemma} 
The construction of the family in \autoref{lemma:instab4d} shares some similarities with some earlier examples due to Perelman \cite{Perelmanbetti} and X. Menguy \cite{Menguy inftop}.

In \autoref{lemma:instab4d} (iii), the distance $\dist$ in the limit cannot be induced by a smooth Riemannian metric by the diffeomorphic stability \cite[Theorem A.1.3]{CheegerColding97I}. However, the underlying topological space of the limit is a topological manifold, and the distance is induced by a smooth Riemannian metric away from two points. We note that $\mathbb{CP}^2\#\overline{\mathbb{CP}^2}$ can be viewed as the nontrivial $S^2$ bundle over $S^2$. Hence $\mathbb{CP}^2\#\overline{\mathbb{CP}^2}$ bounds a smooth $5$-manifold $\overline{M}^5$, namely, the nontrivial $\overline{D}^3$ bundle over $S^2$. \\
Starting from \autoref{lemma:instab4d},  Colding and Naber construct a smooth metric $g$ with $\Ric\ge 0$ and Euclidean volume growth on $M^5$, i.e., the nontrivial $D^3$ bundle over $S^2$, such that the family of cross-sections of blow-downs is 
\begin{equation}
\mathcal{C}_{\infty}(M^5,g)=\left\{(\mathbb{CP}^2\#\overline{\mathbb{CP}^2},\dist_{g_t})\, : \, t\in (0,1]\right\} \bigcup \left\{ (S^4,\dist)\right\}\, .
\end{equation}
In particular, there are two blow-downs with non-homeomorphic cross-sections. Recently, P. Reiser has extended Colding and Naber's construction in \cite{Reiser} and obtained a large family of interesting examples of cross-sections of blow-downs in dimensions $n >4$.

\section{Some previous developments}\label{sec:previous}

In this section, we briefly review some of the numerous previous developments related to \autoref{thm:mainIntro}. We do not claim any completeness, and the choice is only motivated by the author's taste. 

\subsection{Dimensions $n\le 3$}\label{sec:prelle3}

For a complete surface $(M^2,g)$ with $\Ric\ge 0$ (equivalently, with nonnegative Gaussian curvature) and quadratic volume growth, the blow-down is unique and isometric to $C(S^1_{\theta})$, where $S^1_{\theta}$ is a circle of length $0<\theta\le 2\pi$. The length of the circle is determined by the asymptotic volume ratio of $(M^2,g)$. More precisely, it holds
\begin{equation}
\theta=2\lim_{r\to \infty}\frac{\mathrm{vol}(B_r(p))}{r^2}\, .
\end{equation}
\medskip

For $n=3$, the asymptotic can be richer, although it is still quite rigid with respect to higher dimensions and completely understood. Note that blow-downs of smooth, complete $(M^3,g)$ with $\Ric\ge 0$ and Euclidean volume growth need not be unique; see, for instance, \cite[Theorem 1.1]{ColdingNabercones}. 

\begin{theorem}\label{thm:3d}
Let $(M^3,g)$ be smooth, complete, with $\Ric\ge 0$ and Euclidean volume growth. Then every blow-down of $(M^3,g)$ is a metric cone $C(Z)$ over a cross-section $(Z,\dist_Z)$ which is an Alexandrov space with curvature $\ge 1$ with $Z\approx S^2$. Conversely, for every Alexandrov space $(Z,\dist_Z)$ with curvature $\ge 1$ and $Z\approx S^2$ there exist $\delta\le 1$ and $(M^3,g)$ with $\Ric\ge 0$ and Euclidean volume growth with a blow-down isometric to $(Z,\delta\dist_Z)$. 
\end{theorem}

That any cross-section $(Z,\dist_Z)$ of some blow-down of $(M^3,g)$ must be an Alexandrov space with curvature $\ge 1$ follows from the combination of two results. Ketterer's \cite[Theorem 1.1]{Ketterer15}, implies that $(Z,\dist_Z,\mathcal{H}^2)$ is an $\RCD(1,2)$ space. In dimension two, the smooth intuition suggests that Ricci curvature and sectional curvature should be the same. This is confirmed by \cite[Theorem 1.1]{LytchakStadler}, due to A. Lythchak and S. Stadler, showing that $(Z,\dist_Z)$ is a $2$-dimensional Alexandrov space with curvature $\ge 1$.\footnote{More precisely, here we are exploiting the implication from $\RCD$ to Alexandrov, which is only valid in dimension $2$. The converse implication is valid in any dimension, as suggested again by the smooth intuition and proved by A. Petrunin in \cite{PetruninLSV}.}

Note that any $2$-dimensional Alexandrov space is homeomorphic to a surface (possibly with boundary), see \cite[Corollary 10.10.3]{BuragoBuragoIvanov}. If the curvature is $\ge 1$, then relying on a generalized Bonnet-Myers theorem \cite[Theorem 10.4.1]{BuragoBuragoIvanov}, it is possible to check that $Z$ must be homeomorphic to either $\overline{D}^2$, $\mathbb{RP}^2$, or $S^2$.\\
To prove that $Z\approx S^2$, there are at least three morally independent strategies:
\begin{itemize}
\item[i)] We can rely on \cite[Corollar 3.1]{Zhu93}, due to S.-H. Zhu. Since any blow-down $C(Z)$ of $(M^3,g)$ is, in particular, a $3$-dimensional noncollapsed Ricci limit space, it is a homology manifold of dimension $3$. Hence, if $o$ denotes the vertex of $C(Z)$, we have $H_{*}(C(Z),C(Z)\setminus \{o\};\Z)=H_{*}(\R^3,\R^3\setminus\{0\};\Z)$, by definition of homology manifold. On the other hand, for any cone it holds $H_{*}(C(Z),C(Z)\setminus \{o\};\Z)=H_{*}(Z,\Z)$, up to shifting the indexes. In particular, $H_{*}(\R^3,\R^3\setminus\{0\};\Z)=H_{*}(S^2,\Z)$. Therefore, $Z$ must be a homology $2$-sphere. It follows that $Z\approx S^2$. We stress that \cite[Corollary 3.1]{Zhu93} is stated only for limits of sequences with uniformly bounded diameters. The statement should generalize to the noncompact setting without too much effort.\footnote{We also warn the reader that the proof of \cite[Corollary 3.1]{Zhu93} relies on \cite[Theorem, pg. 393]{Petersen90}, which is not correct as stated, see \cite{Moore} and \cite{Ferry94}. However, the statement of \cite[Corollary 3.1]{Zhu93} is correct, and its proof can be fixed by exploiting the fact that noncollapsed Ricci limits of dimension $n$ have Hausdorff dimension $n$.}
\item[ii)] We can rely on the work of Simon \cite{Simon14} and Simon and P. Topping \cite{SimonTopping22}, who used the Ricci flow to prove that any noncollapsed $3$-dimensional Ricci limit space is a topological $3$-manifold (without boundary). Note that a topological $3$-manifold is, in particular, a homology $3$-manifold. One can conclude by arguing as before.
\item[iii)] We can rely on the tools developed in \cite{BruePigatiSemola}. Note that $Z$ cannot have boundary, by Cheeger and Colding's (no) boundary stability \cite[Theorem 6.2]{CheegerColding97I}. Let $p\in M^3$ denote an arbitrary point and $G_p:M\setminus\{p\}\to (0,\infty)$ denote the unique positive Green's function of the Laplacian with pole at $p$ and vanishing at infinity. In \cite{BruePigatiSemola}, we prove that there exists a sequence $t_i\to 0$ such that $t_i$ is noncritical for $G_p$ and the level set $\{G_p=t_i\}$ is homeomorphic to $Z$. Since the surface $\{G_p=t_i\}$ bounds the compact $3$-manifold $\{G_p>t_i\}$, $Z$ cannot be homeomorphic to $\mathbb{RP}^2$. Indeed the boundary of any $3$-manifold has even Euler characteristics.
\end{itemize}
For the proof of the converse implication of \autoref{thm:3d} there are (at least) two different approaches:
\begin{itemize}
\item[i)] Argue as we did in Section \ref{sec:Realization4d} for the $4$-dimensional case. Namely, one can combine: (a) the methods in \cite{ColdingNabercones}; (b) the fact that any sphere $(S^2,\dist)$ with curvature $\ge 1$ in the Alexandrov sense is the Gromov-Hausdorff limit of a sequence of smooth Riemannian metrics on $S^2$ with curvature $\ge c>0$, see \cite[Chapter 7, Section 6]{Alexandrov}; (c) the path connectedness of the space of metrics with curvature $\ge c>0$ on $S^2$ with respect to the smooth topology, established by H. Weyl in 1916 using the uniformization theorem.   
\item[ii)] Start from the observation, due to N. Lebedeva, V Matveev, Petrunin, and V. Shevchishin in \cite[Corollary 3.1]{Lebedevaetal}, that if $(Z,\dist_Z)$ is an Alexandrov space with curvature $\ge 1$ with $Z\approx S^2$ then $C(Z)$ is isometric to the surface of a convex cone in $\R^4$. This relies on the classical embedding theorem of A. D. Alexandrov. Such surface can be written as the graph of a homogeneous convex function $f:\R^3\to \R$, up to rotating the coordinates. Regularizing $f$ by convolution, one obtains a smooth, complete metric $g$ on $\R^3$ (i.e., the induced metric on the graph) with nonnegative sectional curvature and blow-down $C(Z)$. This option has the clear advantage with respect to (i) of producing a metric with nonnegative sectional curvature.
\end{itemize}

\subsection{The Ricci-flat case}
In the Ricci-flat $4$-dimensional case, the large-scale structure is much more rigid. Many people contributed to the problem before \cite{CheegerNaber15}. M. Anderson in \cite[Theorem 3.5]{Anderson89}, and S. Bando, A. Kasue, and H. Nakajima in \cite[Theorem, pg. 314]{BandoKasueNakajima}, obtained the same conclusions as in \autoref{thm:ChNaALE} under the additional assumption that
\begin{equation}
\int_{M^4}|\mathrm{Riem}|^2\di\mathrm{vol}<\infty\, .
\end{equation}
G. Tian obtained similar conclusions in \cite{Tian90}, where the focus was on the Kähler-Einstein case.
\medskip

We review some of the key ideas of the proof of \autoref{thm:ChNaALE}, as they play an important role in the proof of \autoref{thm:mainIntro} as well.
As explained in the introduction of \cite{CheegerNaber15}, it has been understood since \cite{CheegerColding96} that the key step towards proving \autoref{thm:ChNaALE} would be establishing the following:

\begin{theorem}{\cite[Theorem 5.2]{CheegerNaber15}}\label{thm:codim3}
Let $(M^n_i,g_i,p_i)$ be smooth Riemannian manifolds with $|\Ric_i|\to 0$, such that
\begin{equation}
(M^n_i,\dist_{g_i},p_i)\xrightarrow{\mathrm{pGH}}\R^{n-2}\times C(S^1_{\theta})\,, \quad 0<\theta\le 2\pi\, .
\end{equation}
Then $\theta=2\pi$, i.e., $\R^{2}\times C(S^1_{\theta})=\R^n$.
\end{theorem}

We can rephrase \autoref{thm:codim3} by saying that codimension $2$ singularities do not appear for noncollapsed limits with bounded Ricci curvature.
If \autoref{thm:codim3} holds, then one can rule out codimension $3$ singularities as well, exploiting Anderson's $\varepsilon$-regularity theorem from \cite{Anderson90}. Namely, one can prove:

\begin{theorem}{\cite[Theorem 5.12]{CheegerNaber15}}\label{thm:codim4}
Let $(M^n_i,g_i,p_i)$ be smooth Riemannian manifolds with $|\Ric_i|\to 0$, $\mathrm{vol}(B_1(p_i))>v>0$ and 
\begin{equation}\label{eq:pGH3sim}
(M^n_i,\dist_{g_i},p_i)\xrightarrow{\mathrm{pGH}}\R^{n-3}\times C(Y)\, ,\quad i\to \infty\, ,
\end{equation}
for some metric space $(Y,\dist_Y)$. Then $(Y,\dist_Y)$ is the unit $2$-sphere, hence $\R^{n-3}\times C(Y)=\R^n$.
\end{theorem}

\begin{proof}[Sketch of the proof]
Thanks to \autoref{thm:codim3}, all blow-ups of $(Y,\dist_Y)$ are isometric to $\R^2$. By Anderson's $\epsilon$-regularity theorem from \cite{Anderson90}, it follows that $(Y,\dist_Y)$ is a smooth surface endowed with a smooth metric $g$ such that $\Ric_g=g$. Indeed, a cone over a smooth section $(N^{n-1},g_N)$ is Ricci-flat away from the vertex if and only if the cross-section satisfies $\Ric_N=(n-2)g_N$. This means that $Y$ is isometric to either $S^2$ or $\mathbb{RP}^2$, endowed with the round metric with constant curvature $\equiv 1$. 

We need to rule out the second option. The origins of the argument below date back to Cheeger, Colding, and Tian's \cite{CheegerColdingTian}. 

Let $o\in C(Y)$ denote the vertex of the cone. Note that $\R^{n-3}\times C(Y)$ is a smooth Riemannian manifold away from $\mathcal{S}:=\R^{n-3}\times\{o\}$. By \cite{Anderson90} the $\mathrm{pGH}$ convergence in \eqref{eq:pGH3sim} can be upgraded to local $C^{1,\alpha}$ convergence away from $\mathcal{S}$. We consider the coordinate map onto the Euclidean factor $u_{\infty}:\R^{n-3}\times C(Y)\to \R^{n-3}$ and the squared distance from the singular set, $\dist^2_{\mathcal{S}}:\R^{n-3}\times C(Y)\to [0,\infty)$. Note that they satisfy $\Delta u_{\infty}=0$ and $\Delta \dist^2_{\mathcal{S}}=6$. We can approximate them locally uniformly and smoothly away from $\mathcal{S}$ with smooth functions $u_i:M_i\to \R^{n-3}$ and $h_i:M_i\to [0,\infty)$ such that $u_i(p_i)=0$, $\Delta u_i=0$ and $\Delta h_i=6$. By smooth convergence away from $\mathcal{S}$, the level set $\{u_i=0\}\cap \{h_i=1\}$ is diffeomorphic to $Y$ for $i$ large enough. Moreover, it bounds the compact $3$-manifold $\{u_i=0\}\cap \{h_i\le 1\}$. Since the boundary of every $3$-manifold has even Euler characteristic, $Y\approx S^2$. 
\end{proof}

\begin{proof}[Sketch of the proof of \autoref{thm:ChNaALE}]
We can exploit \autoref{thm:codim4} to show that the cross-section of each blow-down must be a smooth Riemannian $3$-manifold. Arguing as we did above, we deduce that any such cross-section is Einstein with Einstein constant $2$. Hence, it is a spherical space form $S^3/\Gamma$ endowed with the round metric with constant curvature $\equiv 1$. The uniqueness of the blow-down easily follows from the connectedness of the space of cross-sections. 
\end{proof}

While \autoref{thm:codim3} is intuitively clear for $n=2$, since the distributional curvature of $C(S^1_{\theta})$ is a Dirac delta at the vertex unless $\theta=2\pi$, no such elementary heuristics can be made to work for $n\ge 3$.
To prove \autoref{thm:codim3} for $n\ge 3$, Cheeger and Naber introduced a powerful \textbf{slicing} theorem, that we formulate below in a simplified way.

\begin{theorem}{\cite[Theorem 1.23]{CheegerNaber15}}\label{thm:slicing}
Let $(M^n_i,g_i,p_i)$ be smooth Riemannian manifolds with $\Ric_i\ge -\varepsilon_i\to 0$, such that
\begin{equation}\label{eq:convn-2}
(M^n_i,\dist_{g_i},p_i)\xrightarrow{\mathrm{pGH}}\R^{n-2}\times \Sigma\, ,
\end{equation}
for some metric space $(\Sigma,\dist_{\Sigma})$. Then there exist harmonic ``almost-splitting''\footnote{See \cite[Section 9]{CheegerFermi} for the relevant background.} maps $u_i:B_2(p_i)\to \R^{n-2}$ and points $q_i\in u_i(B_1(p_i))$ such that for every $x\in u_i^{-1}(q_i)$ and every $0<r<1$ it holds
\begin{equation}
\mathrm{d}_{\mathrm{GH}}\left(B_r(x),B_r(0^{n-2},s)\right)<\delta_ir\, ,
\end{equation}
where $B_r(0^{n-2},s)\subset \R^{n-2}\times \Sigma_{x,r}$, $(\Sigma_{x,r},\dist_{\Sigma_{x,r}})$ is a metric space, and $\delta_i\to 0$ as $i\to\infty$.
\end{theorem}


The slicing \autoref{thm:slicing} provides a good slice, i.e., a level set, such that the almost product structure at the top scale coming from \eqref{eq:convn-2} is kept at all locations and scales on the slice. 


\begin{proof}[Sketch of the proof of \autoref{thm:codim3}]
The proof relies on a blow-up argument in the spirit of \cite{Anderson90}. If the cone angle $\theta$ is $<2\pi$, then the minimum of the harmonic radius $r_h$\footnote{See \cite[Chapter 10]{Petersenbook} for the relevant background.} on the slice $u^{-1}_i(q_i)$ is attained at some $x_i\in u_i^{-1}(q_i)$ and goes to $0$ as $i\to\infty$. We set $r_i:=r_h(x_i)$, rescale and obtain a sequence
\begin{equation}
(M^n_i,r_i^{-1}\dist_{g_i},x_i)\xrightarrow{\mathrm{pGH}}(\R^{n-2}\times \tilde{\Sigma},\dist,x_{\infty})\, ,
\end{equation}
where $\R^{n-2}\times \tilde{\Sigma}$ is a Ricci-flat manifold with Euclidean volume growth. It follows that $\R^{n-2}\times \tilde{\Sigma}=\R^n$ and the convergence is in $C^{1,\alpha}$, by \cite{Anderson90}. Note that the harmonic radius of $(M^n_i,r_i^{-1}\dist_{g_i},x_i)$ at $x_i$ equals $1$ by construction. The harmonic radius behaves continuously in the limit. This results into a contradiction, since $r_h\equiv+\infty$ on $\R^n$. 
\end{proof}

For the above argument, it was clearly crucial that the almost product structure persists at all locations and scales along the good slice.

\begin{remark}\label{rm:slicingfails}
The slicing \cite[Theorem 1.23]{CheegerNaber15}, stated here as \autoref{thm:slicing}, has no exact counterpart in higher codimensions, i.e., when we replace $\R^{n-2}$ at the left-hand side of \eqref{eq:convn-2} with $\R^{n-k}$ for $3\le k\le n-1$. This can be verified by considering the examples constructed by A. Kasue and T. Washio in \cite[pg. 913-914]{KasueWashio}. Indeed, they construct metrics $g$ on $\R^n$ for $n\ge 4$ of the form
\begin{equation}
g:=f^2(r)\di t^2+\di r^2+\eta^2(r)g_{S^{n-2}}\, ,
\end{equation}
with $\Ric\ge 0$, Euclidean volume growth, and such that the following hold: 
\begin{itemize}
\item[i)] the projection onto the $t$-variable is a harmonic function with linear growth; 
\item[ii)] translations in the $t$-direction form a $1$-parameter family of isometries;
\item[iii)] $(\R^4,g)$ does not split any line isometrically.
\end{itemize} 
By (i) and \cite{CheegerColdingMinicozzi}, the projection onto the $t$-variable converges to a splitting function when we blow-down $(\R^4,g)$, i.e., when we restrict it to balls $B_{R_i}(p)$ with $R_i\to\infty$ it induces $\varepsilon_i$-splittings with $\varepsilon_i\to 0$ as $i\to\infty$. However, it cannot have level sets $\{t=c_i\}$ such that $B_r(q)$ becomes closer and closer as $i\to\infty$ to a product at all points in $q\in \{t=c_i\}$ and at every scale $0<r<R_i$, otherwise $(\R^4,g)$ would split, by (ii).
\end{remark}

\begin{remark}
By \cite[Theorem 1.2]{ColdingNabercones}, the conclusion of \autoref{thm:slicing} cannot hold \emph{for every} level set in general for $n\ge 3$. However, it holds for a large measure set of level sets.

For $n\ge 4$, it is possible to produce examples meeting the assumptions of \autoref{thm:slicing} where the functions $u_i$ do have critical points in $B_1(p_i)$ for every $i\in\N$. A possible construction combines the techniques in \cite{ColdingNabercones} with the gluing construction developed by Menguy in \cite{Menguy inftop}, exploiting the earlier work of Perelman \cite{Perelmanbetti}. One should be able to produce such examples so that the balls $B_1(p_i)$ are not contractible inside the balls $B_2(p_i)$.
\end{remark}

\subsection{K\"ahler surfaces}

When $(M^4,g)$ is a K\"ahler surface, \autoref{thm:mainIntro} follows from the works of G. Liu and G. Székelyhidi \cite{LiuSzekeI,LiuSzekeII}, as detailed in the Appendix of \cite{Zhou23}. Many of the techniques employed in these works heavily exploit the K\"ahler structure and, as such, have no counterpart in the setting of \autoref{thm:mainIntro}. 

To give readers an insight as to how the K\"ahler assumption can restrict the geometry of noncollapsed Ricci limits, it is worth mentioning that splittings of Euclidean factors always come into pairs in this setting. This is due to Cheeger, Colding, and Tian in \cite{CheegerColdingTian}. In particular, \autoref{thm:noRP2BPS} below, which is a key step towards the proof of \autoref{thm:mainIntro}, is a straightforward corollary of \cite[Theorem 9.1]{CheegerColdingTian} in the K\"ahler case.

\subsection{Nonnegative sectional curvature}\label{sec:curvge0}
When we strengthen the assumption $\Ric\ge 0$ to $\mathrm{Sect}\ge 0$, much stronger results hold, independently of the dimension. 

The first point worth stressing is that any smooth, complete $(M^n,g)$ with $\mathrm{Sect}\ge 0$ and Euclidean volume growth is diffeomorphic to $\R^n$. This follows from the soul theorem of Cheeger and D. Gromoll, together with the subsequent refinement due to W.-A. Poor, and the fact that the soul of any such manifold is a point. Otherwise, the volume growth could not be Euclidean.
Moreover, the blow-down is always unique in this context, as pointed out by Gromov in \cite[pg. 58-59]{BallmannGromovSchroeder}. See also the work of Kasue \cite{Kasue88} for a detailed proof. The cross-section $(Z,\dist_Z)$ of the blow-down is an Alexandrov space with curvature $\ge 1$, thanks to the work of Y. Burago, Gromov, and Perelman \cite[Corollary 7.10]{BuragoGromovPerelman92}. 

There are two other important statements worth recording:
\begin{itemize}
\item[i)] $Z\approx S^{n-1}$;
\item[ii)] there is a sequence of smooth Riemannian metrics $g_i$ on $S^{n-1}$ with sectional curvature $\ge -1$ such that 
\begin{equation}
(S^{n-1},\dist_{g_i})\xrightarrow{\mathrm{GH}}(Z,\dist_Z)\, , \quad  \text{as $i\to\infty$}\, .
\end{equation}
\end{itemize}
Both i) and ii) follow from the (much more general) results obtained by V. Kapovitch in \cite{Kapovitchcross}. The proofs are based on a slicing method. The slicing uses convex functions rather than harmonic functions in this case. It is crucial that lower bounds on the sectional curvature are inherited by level sets of convex functions.

\subsection{$\rm{CAT}(0)$ four-manifolds}

Although the curvature is assumed to have the opposite sign, we find it worth concluding this overview with the following result due to Lytchak, K. Nagano, and Stadler in \cite{LytchakNaganoStadler}.

\begin{theorem}\label{thm:CAT0}
Let $(X,\dist_X)$ be a $\mathrm{CAT}(0)$ four-manifold, i.e., a $\mathrm{CAT}(0)$ metric space whose underlying topological space is a topological four-manifold. Then $X\approx \R^4$. Moreover, the ideal boundary $\partial_{\infty}X$ is homeomorphic to $S^3$ and the canonical compactification $\overline{X}:=X\cup\partial_{\infty}X$ is homeomorphic to the closed ball in $\R^4$.
\end{theorem} 

The proof of \autoref{thm:CAT0} is also based on a slicing method, partly building on the structure theory for (geodesically complete) $\rm{CAT}(0)$ spaces developed earlier by Lytchak and Nagano in \cite{LytchakNagano19,LytchakNagano22}, as well as on a previous contribution due to P. Thurston \cite{Thurston96}. A second similarity with the proof of \autoref{thm:mainIntro} is the use of a series of deep results in geometric topology. 

Statements analogous to \autoref{thm:CAT0} hold in any dimension less than four. On the other hand, they fail in higher dimensions, as originally shown by M.-W. Davis and T. Januszkiewicz in \cite{DavisJanuskiewicz}. 

\section{On the proof of \autoref{thm:mainIntro}}\label{sec:proof}

The proof of \autoref{thm:mainIntro} depends on three independent tools, all developed in \cite{BruePigatiSemola}. We introduce them below and explain how they lead to \autoref{thm:mainIntro}. In the forthcoming sections we discuss the ideas that enter into their proofs.
\medskip

 The first tool is a variant of \autoref{thm:codim4} tailored for lower Ricci bounds:
 
 \begin{theorem}{\cite[Theorem 1.6]{BruePigatiSemola}}\label{thm:noRP2BPS}
 Fix $n\ge 3$. Let $(M^n_i,g_i,p_i)$ be smooth Riemannian manifolds with $\Ric_i\ge -\delta_i$, $\delta_i\to 0$, $\mathrm{vol}(B_1(p_i))>v>0$ and 
\begin{equation}\label{eq:pGH3sim}
(M^n_i,\dist_{g_i},p_i)\xrightarrow{\mathrm{pGH}}\R^{n-3}\times C(Y)\, ,\quad i\to \infty\, ,
\end{equation}
for some metric space $(Y,\dist_Y)$. Then $Y\approx S^2$.
 \end{theorem}
\begin{remark}
It is worth pointing out that \autoref{thm:noRP2BPS} is original only for $n\ge 4$. For $n=3$, it can be obtained by arguing as in the proofs of \autoref{thm:3d} that we outlined above. However, both the approach based on \cite{Zhu93}, as well as the one based on Ricci flow from \cite{SimonTopping22a}, fail to extend to $n\ge 4$, as of today.\\
If the $M_i$'s are assumed to be all orientable and the distance on $Y$ is induced by a smooth Riemannian metric, then the statement follows also from S. Honda's \cite{Hondaor}, in any dimension. Thanks to the recent work of C. Brena, Bru\`e, and Pigati \cite{BrenaBruePigati}, the smoothness assumption on the distance on $Y$ can be dropped.
\end{remark}

The ``no $C(\mathbb{RP}^2)$'' \autoref{thm:noRP2BPS} prevents the appearance in $\mathcal{C}_{\infty}$ of $\RCD(2,3)$ spaces which are clearly not topological manifolds, such as the spherical suspension over a round $\mathbb{RP}^2$. 
It turns out that understanding blow-ups is sufficient to establish topological regularity due to the following \textbf{manifold recognition} theorem for $3$-dimensional $\RCD$ spaces:

\begin{theorem}{\cite[Theorem 1.8]{BruePigatiSemola}}\label{thm:manrecRCD}
Let $(X,\dist,\haus^3)$ be an $\RCD(-2,3)$ space. Then $X$ is a topological $3$-manifold if and only if the cross-section of every blow-up of $(X,\dist)$ at every point is homeomorphic to $S^2$.
\end{theorem}

We can rephrase \autoref{thm:manrecRCD} by saying that for an $\RCD(-2,3)$ space $(X,\dist,\haus^3)$ the following are equivalent:
\begin{itemize}
\item[i)] every $x\in X$ has a neighbourhood homeomorphic to $\R^3$;
\item[ii)] all blow-ups of $(X,\dist)$ at all points are homeomorphic to $\R^3$.  
\end{itemize}
For every $x\in X$, all the blow-ups of $(X,\dist)$ at $x$ have cross-sections homeomorphic to each other, thanks to \cite{LytchakStadler}. In particular, they are all homeomorphic to each other. Hence, ii) is equivalent to asking that one blow-up of $(X,\dist)$ at each point is homeomorphic to $\R^3$. Moreover, it is clear from the proof that \autoref{thm:manrecRCD} can be localised to open sets.

\begin{remark}\label{rm:nomanrec4}
A characterisation of the local topology in terms of blow-ups is not possible for $\RCD(-(n-1),n)$ spaces $(X,\dist,\haus^n)$ when $n\ge 4$. Indeed, Menguy constructed in \cite[Theorem 0.6]{Menguy inftop} a noncollapsed $4$-dimensional Ricci limit space $(Z,\dist_Z)$ which is a smooth Riemannian manifold away from a single point $z\in Z$ and such that no neighbourhood of $z$ has finite topological type. All the blow-ups of $(Z,\dist_Z)$ at $z$ (and hence at every point in $Z$) are homeomorphic to $\R^4$. 

On the other hand, it seems conceivable that if $(X,\dist,\haus^4)$ is an $\RCD(-3,4)$ space and $X$ is a topological $4$-manifold then all the cross-sections of all blow-ups of $(X,\dist)$ should be homeomorphic to $S^3$.
\end{remark}


The last tool for the proof of \autoref{thm:mainIntro} is a topological stability theorem.

\begin{theorem}{\cite[Theorem 1.11]{BruePigatiSemola}}\label{thm:stability3d}
Let $(X_i,\dist_i,\haus^3)$ be $\RCD(-2,3)$ spaces such that $X_i$ is a topological $3$-manifold for every $i\in \N$. Assume that 
\begin{equation}\label{eq:limstab}
(X_i,\dist_i)\xrightarrow{\mathrm{GH}}(X,\dist)\, \quad \text{as $i\to\infty$}\, ,
\end{equation}
without collapse, for some compact $\RCD(-2,3)$ space $(X,\dist,\haus^3)$. Then there exists $i_0\in\N$ such that $X_i$ is homeomorphic to $X$ for every $i\ge i_0$.
\end{theorem}

\begin{remark}
When the $(X_i,\dist_i)$ are all smooth Riemannian $3$-manifolds \autoref{thm:stability3d} was obtained in \cite{Simon14}, based on Ricci flow. 

\end{remark}

We can view \autoref{thm:stability3d} as a counterpart of Perelman's stability theorem, see \cite{Perelman99} and the exposition in \cite{KapovitchPer}, in the context of synthetic lower Ricci curvature bounds. There are two key differences with the case of Alexandrov spaces with curvature bounded from below:
\begin{itemize}
\item[i)] the restriction on the dimension being (less or equal than) $3$; 
\item[ii)] the restriction on the spaces being topological manifolds.
\end{itemize} 
The restriction on the dimension is unavoidable in this context. Indeed, Anderson constructed in \cite{Andersonlarge} a sequence of $(M^4_i,g_i)$ with $\Ric_i\ge 3$, $\mathrm{vol}(M_i)\ge v>0$, and $\diam(M_i)\le D<\infty$ for each $i\in\N$ and such that 
\begin{equation}
(M^4_i,g_i)\xrightarrow{\mathrm{GH}}(X,\dist)\, \quad \text{as $i\to \infty$}\, ,
\end{equation}
but for no $i\in \N$ the manifold $M^4_i$ is homeomorphic to $X$. See also the work of Y. Otsu \cite{Otsu} for some similar constructions and note that the failure of this topological stability is at the heart of \autoref{lemma:instab4d} as well. 

\begin{remark}
The topological stability does not hold even for noncollapsing sequences of Einstein manifolds when $n\ge 4$. For instance, C. LeBrun and M. Singer constructed in \cite{LeBrunSinger} sequences of Ricci flat metrics $g_i$ on a K3 surface Gromov-Hausdorff converging without collapse to the flat orbifold $T^4/\iota$, where $\iota$ is an involution of $T^4$ acting by reflection on each side. The existence of such sequences had been suggested earlier by G.-W. Gibbons and C.-N. Pope, and by D.-N. Page, in the late Seventies. 
\end{remark}
Concerning ii), it seems conceivable that \autoref{thm:stability3d} generalises to any noncollapsing sequence of $\RCD(-2,3)$ spaces $(X_i,\dist_i,\haus^3)$, although this generalisation requires some new ideas; see Section \ref{sec:OpenQuest} for a thorough discussion.

\begin{proof}[Sketch of the proof of \autoref{thm:mainIntro}]
We follow the proof of \cite[Theorem 12.14]{BruePigatiSemola}. Let $(Z,\dist_Z)\in \mathcal{C}_{\infty}$ be the cross-section of a blow-down of $(M^4,g)$. We noted in Section \ref{sec:Intro} that $(Z,\dist_Z,\haus^3)$ is an $\RCD(2,3)$ space. 
\medskip

\textbf{Claim 1}: $Z$ is a topological $3$-manifold.\\
Thanks to \autoref{thm:manrecRCD} it suffices to prove that if $z\in Z$ and $C(W)$ is a blow-up of $(Z,\dist_Z)$ at $z$, then $W\approx S^2$.
Note that, if $C(W)$ is a blow-up of $(Z,\dist_Z)$ at $z$, then $\mathbb{R}\times C(W)$ is a blow-up of $C(Z)$ at $(1,z)$. Since $C(Z)$ is a noncollapsed Ricci limit space (it is a blow-down of $(M^4,g)$), $\mathbb{R}\times C(W)$ is a noncollapsed Ricci limit space as well by a standard diagonal argument. By \autoref{thm:noRP2BPS} $W\approx S^2$. 
\medskip

\textbf{Claim 2}: $Z$ is homeomorphic to a spherical space form.\\
It suffices to show that $Z$ is compact with finite fundamental group. The resolution of the elliptization conjecture due to Perelman then yields the homeomorphism with a spherical space form.
The compactness of $(Z,\dist_Z)$ follows from the generalized Bonnet-Myers theorem for $\CD(n-1,n)$ spaces, due to K.-T. Sturm \cite{Sturm07I} and independently J. Lott and C. Villani \cite{LottVillani}. The finiteness of the fundamental group follows from the generalized Bonnet-Myers as well. Indeed, A. Mondino and G. Wei proved in \cite[Theorem 1.1]{MondinoWei} that the universal cover of an $\RCD(n-1,n)$ space is an $\RCD(n-1,n)$ space.
\medskip

\textbf{Claim 3:} The homeomorphism type is fixed on $\mathcal{C}_{\infty}$.  \\
The claim follows from the connectedness of $\mathcal{C}_{\infty}$ with respect to the Gromov-Hausdorff topology, \textbf{Claim 1}, and the stability \autoref{thm:stability3d}. 
\end{proof}

\subsection{The ``no $C(\mathbb{RP}^2)$'' \autoref{thm:noRP2BPS}}
We can view \autoref{thm:noRP2BPS} as a counterpart for \autoref{thm:codim4}, with the characterization of the section $(Y,\dist_Y)$ up to isometry being replaced by a characterization up to homeomorphism. Again, \autoref{thm:3d} shows that \autoref{thm:noRP2BPS} is optimal in this setting. 

\medskip

The proof starts from the same slicing idea of the proof of \autoref{thm:codim4}, and it hinges on a variant of \autoref{thm:slicing}. The key difference with the bounded Ricci case dealt with in \cite{CheegerNaber15} is the need to rule out topological codimension $3$ singularities without being able to rule out metric codimension $2$ ones. Indeed, \autoref{thm:codim3} fails when we drop the upper Ricci curvature bound from the assumptions. 
\medskip

To begin, it is worth introducing a variant of \autoref{thm:slicing}, which is better suited for our purposes.

\begin{theorem}{\cite[Theorem 5.2]{BruePigatiSemola}}\label{thm:slicingBPS}
 Fix $n\ge 3$. Let $(M^n_i,g_i,p_i)$ be smooth Riemannian manifolds with $\Ric_i\ge -\delta_i$, $\delta_i\to 0$, $\mathrm{vol}(B_1(p_i))>v>0$ and 
\begin{equation}\label{eq:pGH3simb}
(M^n_i,\dist_{g_i},p_i)\xrightarrow{\mathrm{pGH}}\R^{n-3}\times C(Y)\, ,\quad i\to \infty\, ,
\end{equation}
for some metric space $(Y,\dist_Y)$. Then, we can approximate the function 
\begin{equation}\label{eq:slicinglimit}
(\pi_{\R^{n-3}},\dist_{\R^{n-3}\times \{o\}}(\cdot)): \R^{n-3}\times C(Y)\to \R^{n-3}\times[0,\infty)
\end{equation}
in the uniform and $W^{1,2}$ sense with a sequence of (smooth) functions
\begin{equation}\label{eq:slicingfunctions}
(v_i,u_i): B_{10}(p_i)\to \R^{n-3}\times[0,\infty)\, ,
\end{equation}
such that the following holds. There exist $\epsilon_i>0$ with $\epsilon_i\to 0$ as $i\to\infty$ and Borel sets $\mathcal{B}_i\, \subset B_1(0^{n-3})\times [0,10]$
	such that
	\begin{itemize}
		\item[(i)] 	$\mathcal{L}^{n-2}(\mathcal{B}_i)\le\epsilon_i $;

		\item[(ii)] for every $(x,y) \in  B_1(0^{n-3}) \times [0,10]$ with
		\begin{equation}
			8 \le \sqrt{|x|^2 + y^2}\le 9 \, , \quad
			(x,y)\notin \mathcal{B}_i \, ,
		\end{equation}
		the corresponding level set $\{(v_i,u_i)=(x,y)\}$ is not empty;

		\item[(iii)] for every $s\in(0,c(\epsilon_i,n,v))$ and every $q\in (B_{9}(p_i)\setminus \overline{B}_{8}(p_i))\cap\{|v_i|< 1\}$ with $(v_i,u_i)(q)\notin\mathcal{B}_i$,
		there exists an $(n-2)\times(n-2)$ matrix $L_{q,s}$ such that
		\begin{equation}
			L_{q,s}\circ (v_i,u_i):B_s(q)\to\R^{n-2} 
		\end{equation}
		is an $\epsilon_i$-splitting map.\footnote{Note that we drop the harmonicity from the usual conditions in the definition of $\epsilon$-splitting maps, see \cite[Definition 1.20]{CheegerNaber15}).}
		\end{itemize}
\end{theorem}

The proof of \autoref{thm:slicingBPS} follows the same strategy as the proof of \cite[Theorem 1.23]{CheegerNaber15}, stated here as \autoref{thm:slicing}, although there are some additional error terms that are more delicate to control.

For $n=3$, the functions $u_i$ in \eqref{eq:slicingfunctions} are obtained starting from Green's functions of the Laplacian with constant Dirichlet boundary conditions on $B_{10}(p_i)$ and poles at $p_i$. Note that on the limit cone $C(Y)$ the Green's function of the Laplacian with pole at the vertex $o$ is a power of the distance function $\dist_o(\cdot)$. See \cite[Section 4]{BruePigatiSemola} for the details. \\
The Green's function of the Laplacian has played an important role in several previous works on manifolds and spaces with Ricci bounded from below; see, for instance, \cite{ColdingGreen,JiangNaber}.

\begin{remark}
There is some freedom with the choice of the approximating functions $u_i$ and $v_i$ in \autoref{thm:slicingBPS}. This is important for the applications. 
\end{remark}

\begin{remark}\label{rm:noncrit}
Item (iii) implies that every ``good slice'' contains only noncritical points, see \cite[Section 7.5]{CheegerJiangNaber} for the key idea of the proof. In particular, any such slice is an embedded surface inside the ambient $M^n_i$. Of course, the existence of many noncritical level sets could be achieved with a much more standard argument based on Sard's lemma. The crucial aspect of \autoref{thm:slicingBPS} is that a suitable nondegeneracy of the maps $(v_i,u_i)$ can be obtained at all scales along the good slices.
\end{remark}

The level sets of $(\pi_{\R^{n-3}},\dist_{\R^{n-3}\times \{o\}}(\cdot)): \R^{n-3}\times C(Y)\to \R^{n-3}\times[0,\infty)$, endowed with the induced metric, are (almost) all isometric to $(Y,\dist_Y)$, up to scaling. Hence, the level sets of the approximating functions in \eqref{eq:slicingfunctions} should approximate well $(Y,\dist_Y)$. A potential issue related to the very nature of Gromov-Hausdorff convergence is that these approximations might be very good at scales comparable to $1$ without having any control when we zoom up at much smaller scales. 
The slicing \autoref{thm:slicingBPS} fixes this issue. Namely, we can find many values (cf. with (i) above) such that the corresponding level sets of the functions $(v_i,u_i)$ converge to the section $(Y,\dist_Y)$ in the Gromov-Hausdorff sense, and the ambient manifolds have $(n-2)$ almost splitting directions at all locations and scales on the level (cf. with (iii) above).
\medskip

If $n=3$, then \autoref{thm:slicingBPS} holds also in the non-smooth case where the manifolds $(M^3,g_i)$ in \eqref{eq:pGH3sim}
 are replaced by $\RCD(-\delta_i,3)$ spaces $(X_i,\dist_i,\haus^3)$. This is key to the proof of the manifold recognition \autoref{thm:manrecRCD}. A precise and more effective statement follows.

\begin{proposition}{\cite[Proposition 9.4]{BruePigatiSemola}}
\label{prop:Greensphere}		
        For every $\eta>0$, if $\eps \le \eps_0(\eta,v)$ the following holds.
        Let $(X,\dist,\haus^3)$ be an $\RCD(-\eps,3)$ space with empty boundary (see \cite{DePhilippisGigli2,KapovitchMondino,BrueNaberSemolabdry} for some background about the notion of boundary in this setting) and $p\in X$ be such that $\haus^3(B_1(p))\ge v$.
        Assume that
		\begin{equation}
		\dist_{\rm{GH}}\left(B_{20}(p),B_{20}(o)\right)\le \eps  \, ,
        \quad o \in C(Z)\, ,
		\end{equation}
		where $(Z,\dist_Z)$ is a $2$-dimensional Alexandrov surface with curvature $\ge 1$. Then there exist a good ``Green distance'' $b_p: B_2(p)\to \R$ and
		a Borel set of radii $r\in(1/4,1)$ of measure at least $(1-\eta)\frac34$ such that
		the topological boundary $\mathbb{S}_r(p)$ of the sub-level set $B_{(1-\eta)r}(p)\subset \mathbb{B}_r(p):=\{b_p<r\} \subset B_{(1+\eta)r}(p)$ satisfies:
		\begin{itemize}
			\item[(i)] $\mathbb{S}_r(p) = \{b_p=r\}$;
			\item[(ii)] up to rescaling by $r$, $\mathbb{S}_r(p)$ (endowed with the restriction of $\dist$) is $\eta$-GH-close to $(Z,\tilde{\dist}_Z)$, where $\tilde{\dist}_Z:=2\sin(\dist_Z/2)$;
			\item[(iii)] $\mathbb{S}_r(p)$ is a closed topological surface;
\item[(iv)] for every $q\in \mathbb{S}_r(p)$ and every $0<s<c(\eta,\nu)$,
\begin{equation}
\frac{b_p-r}{\fint_{B_s(q)}|\nabla b_p|}:B_s(q)\to \R\, 
\end{equation}
is an $\eta$-splitting function.
		\end{itemize}
\end{proposition}


\begin{remark}
Even in the case where $(X,\dist)$ is smooth Riemannian, it is not clear whether the Green distance $b_p$ induces a (trivial) fibration on an annulus. This is a major difference with respect to the context of lower sectional curvature bounds. The proof of \autoref{thm:manrecRCD} would vastly simplify in the presence of such a fibration result. It would be interesting to see whether any counterexample exists.  
\end{remark}

The next step is to turn the metric control along the good slices coming from \autoref{thm:slicingBPS} (iii) (or \autoref{prop:Greensphere} (iv)) into topological information.
The starting point is to deal with the rigid case, which arises through a blow-up argument in the spirit of the proof of \autoref{thm:codim4}. 

\begin{lemma}\label{lemma:Sigmaplane}
Fix $n\ge 2$. Let $(M^n_i,g_i,p_i)$ be smooth Riemannian manifolds with $\Ric_i\ge -\delta_i\to 0$, such that
\begin{equation}\label{eq:convn-2EVG}
(M^n_i,\dist_{g_i},p_i)\xrightarrow{\mathrm{pGH}}\R^{n-2}\times \Sigma\, ,
\end{equation}
for some metric space $(\Sigma,\dist_{\Sigma})$. Assume that $\Sigma$ has quadratic volume growth when endowed with the measure $\mathcal{H}^2$. Then $\Sigma$ is an Alexandrov space with $\mathrm{Sect}\ge 0$ homeomorphic to $\R^2$.
\end{lemma}

\begin{proof}[Sketch of the proof]
By the stability of synthetic Ricci curvature lower bounds, $\R^{n-2}\times \Sigma$ is an $\RCD(0,n)$ space when endowed with the Hausdorff measure $\haus^{n}$. By the stabilty of the $\RCD$ condition under splittings, see \cite{Gigli13}, $(\Sigma,\dist_{\Sigma},\haus^{2})$ is an $\RCD(0,2)$ space. By \cite[Theorem 1.1]{LytchakStadler}, $(\Sigma,\dist_{\Sigma})$ is an Alexandrov space with $\mathrm{Sect}\ge 0$. Moreover, $(\Sigma,\dist_{\Sigma})$ has empty boundary by \cite[Theorem 6.2]{CheegerColding97I}; hence it is a topological surface. The quadratic volume growth then yields the homeomorphism with $\R^{2}$. 
\end{proof}

Our goal is to combine \autoref{thm:slicingBPS} (and \autoref{prop:Greensphere}) with \autoref{lemma:Sigmaplane} to show that when we zoom up enough on each good slice, we see no topology. Moreover, there should be a quantitative estimate of what ``enough'' means.
It is key for this aim that the homeomorphism between $\Sigma$ and $\R^2$ arising from \autoref{lemma:Sigmaplane} is quantitatively well-behaved. To clarify what this means, we recall:

 \begin{definition}[Local uniform contractibility] 
 Let $\mathcal{F}$ be a family of metric spaces. Let $C_0,\rho_0>0$ be fixed. We say that the family $\mathcal{F}$ is locally $(C_0,\rho_0)$-contractible provided that for every $(X,\dist)\in\mathcal{F}$, every $x\in X$ and every $0<r<\rho_0$ the inclusion $B_r(x)\to B_{C_0r}(x)$ is homotopic to a constant map in $B_{C_0r}(x)$. 
 \end{definition}

K. Grove and P. Petersen proved in \cite{GrovePetersen88} that the collection $\mathcal{F}_{v,n}$ of all Riemannian manifolds $(M^n,g)$ with sectional curvature $\ge -1$ and volume of each unit ball $\ge v$ is locally linearly $(C_0(n,v),\rho_0(n,v))$-contractible. See also \cite{Petersen90}. By a scaling argument, it follows that in the class of complete $(M^n,g)$ with $\mathrm{Sect}\ge 0$ and Euclidean volume growth, one can take $\rho_0=+\infty$ and $C_0$ depending only on the asymptotic volume ratio. This applies also to the limit slices $(\Sigma,\dist_{\Sigma})$ arising in \autoref{lemma:Sigmaplane}, although they only have $\mathrm{Sect}\ge 0$ in the Alexandrov sense. See Perelman's \cite{Perelman93} for a more general statement. 
\medskip

We can combine the linear contractibility of Alexandrov surfaces with $\mathrm{Sect}\ge 0$ and quadratic volume growth with \autoref{lemma:Sigmaplane} and \autoref{thm:slicingBPS} (iii). The result is that the good slices of $(v_i,u_i)$ in the $M_i$'s are locally uniformly linearly contractible along the sequence.

\begin{proposition}{\cite[Proposition 7.1]{BruePigatiSemola}}\label{prop:unifcontrslices}
Under the same assumptions and with the same notation of \autoref{thm:slicingBPS}, there exist $C_0=C_0(v)>0$, $\rho_0=\rho_0(v)>0$ and $i_0\in\N$ such that for every $i\ge i_0$ and for every $(x,y) \in  B_1(0^{n-3}) \times [0,10]$ with
		\begin{equation}
			8 \le \sqrt{|x|^2 + y^2}\le 9 \, , \quad
			(x,y)\notin \mathcal{B}_i \, ,
		\end{equation}
		the level set $\{(v_i,u_i)=(x,y)\}\subset M_i$ endowed with the restriction of the ambient distance of $M_i$ is a locally $(C_0,\rho_0)$-contractible closed surface.
\end{proposition}

The rough idea of the proof of \autoref{prop:unifcontrslices} is that the failure of local uniform contractibility can be ruled out by a blow-up argument, exploiting \autoref{lemma:Sigmaplane}. 

\medskip

An important insight due to Petersen in \cite{Petersen90} is that on a locally uniformly contractible family of finite dimensional metric spaces, the homotopy type is stable under Gromov-Hausdorff convergence. With the help of these tools, the completion of the proof of \autoref{thm:noRP2BPS} is virtually identical to that of \autoref{thm:codim4}. 

\begin{proof}[Sketch of the proof of \autoref{thm:noRP2BPS}]
Any sequence of good slices $\{(v_i,u_i)=(x_i,y_i)\}$, endowed with the restrictions of the ambient distances of $(M_i,g_i)$ converges to the section $(Y,\dist_Y)$, up to subsequences and scaling. By \autoref{prop:unifcontrslices}, such slices are closed surfaces homotopically equivalent to $Y$ for every sufficiently large $i\in \N$. Moreover, we can assume without loss of generality that $x_i$ is a regular value of $v_i$ for each $i\in \N$. Hence the slice $\{(v_i,u_i)=(x_i,y_i)\}$ bounds the embedded $3$-manifold $\{v_i=x_i\}\cap\{u_i\le y_i\}$. Therefore the slices cannot be homeomorphic to $\mathbb{RP}^2$. Hence, $Y\approx S^2$.
\end{proof}

Arguing in a similar way, we can prove the following:

\begin{proposition}{\cite[Proposition 9.4]{BruePigatiSemola}}\label{prop:loccont3d}
Under the same assumptions and with the same notation of \autoref{prop:Greensphere} the following hold:
\begin{itemize}
	\item[(i)] there exist $\rho_0=\rho_0(v)>0 0$ and $\overline C = \overline C(v)>0$ such that, whenever $q\in\mathbb{S}_r(p)$ and $0< s\le\rho_0 r$, $B_s(q)\cap \mathbb{S}_r(p)$ is 2-connected in $B_{\overline C s}(q)\cap\mathbb{S}_r(p)$;
		\item[(ii)] $\mathbb{S}_r(p)\approx Z$. Hence they are both either homeomorphic to $S^2$ or to $\mathbb{RP}^2$.
	\end{itemize}
\end{proposition}

\subsection{The manifold recognition \autoref{thm:manrecRCD}}

Unfortunately, in the context of \autoref{thm:manrecRCD}, we are not able to construct homeomorphisms between blow-ups and neighbourhoods of points directly. Luckily, this is not the only circumstance where one would like to prove that some metric space is a topological manifold without being able to construct manifold charts directly. We address the reader to the survey papers of J.-W. Cannon \cite{Cannon}, D. Repovš \cite{Repovs}, and A. Cavicchioli, Repovš, and T.-L. Thickstun \cite{CavRepThi}, for an enlightening discussion about the manifold recognition problem, of which \autoref{thm:manrecRCD} is an instance.
 \medskip
 
 The most delicate implication in the proof of \autoref{thm:manrecRCD} is the one from the assumption on the blow-ups to the manifold regularity of $X$. The roadmap for that is as follows:
\begin{itemize}
\item[i)] establish the local (uniform) simple-connectedness of $(X,\dist)$;
\item[ii)] establish the local (uniform) contractibility of $(X,\dist)$;
\item[iii)] prove that $X$ is a \textbf{generalized $3$-manifold};
\item[iv)] prove that $X$ is a topological $3$-manifold.
\end{itemize}

\begin{definition}[Generalized manifold]\label{def:genman}
Let $(X,\dist)$ be a metric space. We say that $X$ is a \emph{generalized $n$-manifold} if it is locally compact, locally contractible, finite-dimensional (in the sense of the covering dimension) and it has the local relative homology of $\setR^n$, i.e., the groups $H_{*}(X,X\setminus\{x\};\Z)$ are isomorphic to $H_{*}(\setR^n,\setR^n\setminus\{0\};\Z)$ for all $x\in X$. 
\end{definition}

Clearly, (i), (ii), and (iii) are necessary conditions for being a topological $3$-manifold. The proof, however, is designed in such a way that to complete each step it is necessary that we have already completed the previous ones. We will discuss each step separately below, with the help of some toy models. 

The starting point for the proof is that any $\RCD(-2,3)$ space is already known to be a manifold on a \emph{large} dense set. We will exploit the assumption on the blow-ups to see that taking the completion does not mess up the manifoldness.
 
Before we move to the outline, it is worth making the above sentence more precise. Kapovitch and A. Mondino proved in \cite[Theorem 1.7]{KapovitchMondino} that for any $\RCD(-(n-1),n)$ space $(X,\dist,\haus^n)$ with empty boundary the following holds. If we let 
\begin{equation*}
 \mathcal{R}_{\eps}:=\{ x\in X\, :\, \dist_{\mathrm{pGH}}(\R^n,(Y,\dist_Y,y))<\eps\,\,  \text{if}\,\,   (Y,\dist_Y,y)\in \mathrm{Tan}_x(X,\dist) \}\, ,
\end{equation*}
then, for $0<\eps<\eps(n)$, $\mathcal{R}_{\eps}$ is open and dense in $X$, $\mathrm{dim}_{\haus}(X\setminus \mathcal{R}_{\eps})\le n-2$, and $\mathcal{R}_{\eps}$ is biH\"older homeomorphic to a smooth Riemannian manifold. The proof is based on two ingredients: 
\begin{itemize}
\item the $\eps$-regularity theorem for spaces with Ricci bounded below, originally due to Colding, and Cheeger and Colding for smooth manifolds and Ricci limits, see \cite[Theorem 9.67]{CheegerFermi}, and generalized to $\RCD$ spaces by Kapovitch and Mondino in \cite{KapovitchMondino};
\item the metric Reifenberg theorem, due to Cheeger and Colding in \cite[Theorem A.1.1]{CheegerColding97I}.
\end{itemize}
For $n=3$, we see that there is a gap of one dimension between the topological regularity obtained with these techniques and the full topological regularity we aim at. In particular, there are certainly examples where $X\setminus \mathcal{R}_{\eps}$ is a $1$-dimensional set, and the metric Reifenberg theorem does not apply therein.

\begin{remark}
It is expected that $X\setminus \mathcal{R}_{\eps}$ is $1$-rectifiable with locally finite $\haus^1$-measure for every $\RCD(-2,3)$ space $(X,\dist,\haus^3)$. This is known, thanks to the work of Cheeger, W. Jiang, and Naber \cite{CheegerJiangNaber}, when $(X,\dist,\haus^3)$ is a noncollapsed Ricci limit, or the cross-section of a blow-down as in \autoref{thm:mainIntro}. However, this information would not simplify the present proof of \autoref{thm:manrecRCD}.
\end{remark}

Given an $\RCD(-2,3)$ space $(X,\dist,\haus^3)$, we define the \textbf{non-manifold set} $\mathcal{S}_{\rm{top}}(X)=\mathcal{S}_{\rm{top}}$ as the set of those points in $X$ with no neighbourhood homeomorphic to $\R^3$. Clearly, $\mathcal{S}_{\rm{top}}$ is a closed subset of $X$. The discussion above can be rephrased by saying that $\mathcal{S}_{\rm{top}}\subset X\setminus\mathcal{R}_{\eps}$ has Hausdorff dimension at most $1$. We will argue that $\mathcal{S}_{\rm{top}}=\emptyset$ provided that every blow-up of $(X,\dist)$ is homeomorphic to $\R^3$. 

\subsection{Local simple-connectedness}\label{subsec:loc1c} 
 
The first step of the proof of \autoref{thm:manrecRCD} is to establish the following:

\begin{proposition}\label{prop:1conn}
Let $v>0$. There exist $C=C(v)>0$ and $\rho=\rho(v)>0$ such that the following holds. If $(X,\dist,\haus^3)$ is an $\RCD(-2,3)$ space with $\haus^3(B_1(p))\ge v$ for any $p\in X$ and such that the cross-sections of all blow-ups are homeomorphic to $S^2$, then $B_r(p)$ is $1$-connected inside $B_{Cr}(p)$ for every $r\le\rho$ and every $p\in X$.
\end{proposition} 
 
\begin{remark}
J. Wang proved in \cite{WangRCD} that every $\RCD$ space $(X,\dist,\meas)$ is semi-locally simply connected. It is an open question whether $\RCD$ spaces are locally simply connected in general.
\end{remark} 
 
The key idea of the proof of \autoref{prop:1conn} is better illustrated with the help of a toy model.

\begin{proposition}\label{prop:global1con}
Let $(M^3,g)$ be smooth, complete with $\Ric\ge 0$ and Euclidean volume growth. Then $M$ is simply-connected. 
\end{proposition}
 
\begin{proof}[Sketch of the proof]
We fix $p\in M$ and let $G:M\setminus\{p\}\to (0,\infty)$ be the (unique) positive Green's function of the Laplacian with pole at $p$, i.e., $G$ solves $\Delta G=-\delta_p$ on $M$ in the sense of distributions. We can apply the slicing \autoref{thm:slicingBPS} and \autoref{prop:unifcontrslices} to the function $G^{-1}$ (suitably rescaled and normalized) along any sequence $(M,r_i^{-2}g,p)$ converging to some blow-down $C(Y)$ of $(M,g)$. Arguing as in the proof of \autoref{thm:noRP2BPS}, we obtain a sequence $s_i\to 0$ such that the level sets $\{G=s_i\}$ are all homeomorphic to $Y$. Since the level sets bound the respective super level sets, they are all homeomorphic to $S^2$.

Note that $M$ has finite $\pi_1$, as shown by P. Li \cite{Lilarge}, and, independently, Anderson \cite{Andersonpi1}. We consider the universal cover $\pi:\overline{M}\to M$, where $\pi_1(M)$ acts by deck transformations. Such universal cover has $\Ric\ge 0$ and Euclidean volume growth as well. Since $\pi_1(M)$ is finite, we can apply the same slicing argument as before to $\overline{G}:=G\circ\pi:\overline{M}\to (0,\infty)$ and check that the level sets $\{\overline{G}=s_i\}$ are homeomorphic to the cross-section of the blow-downs of $\overline{M}$. For the same reasons as above, this implies that they are homeomorphic to $S^2$. Since $\pi_1$ acts by deck transformations on each $\{\overline{G}=s_i\}$ with quotient $\{G=s_i\}$, $\pi_1$ must be trivial, i.e., $M$ is simply connected.
\end{proof} 

\begin{remark}
The conclusion of \autoref{prop:global1con} is well known. However, the original proof in \cite{Zhu93} argues that the universal cover is contractible in the first place. This is enough to rule out the torsion in the fundamental group for topological reasons and, hence, to complete the proof. To prove the contractibility of the universal cover, the idea, originally due to R. Schoen and S.-T. Yau in \cite[Section 2, Lemma 2]{YauSeminar82}, is to combine the sphere theorem in $3$-manifolds topology with the Cheeger-Gromoll splitting theorem. This strategy heavily relies on the information that $M$ is a (topological) $3$-manifold. As such, it is not feasible for our roadmap.
\end{remark}

The argument that we sketched above is robust in the sense that it uses very little of the regularity of $M$, and it can be localized. Namely we can use it to prove the following: 

\begin{proposition}{\cite[Proposition 9.15]{BruePigatiSemola}}\label{prop:conS21c}
There exists $\eps>0$ such that if $(X,\dist,\haus^3)$ is an $\RCD(-\eps,3)$ space, $p\in X$ and 
\begin{equation}
\dist_{\mathrm{GH}}\left(B_{20}(p),B_{20}(o)\right)<\eps\, ,
\end{equation}
where $B_{20}(o)$ is the ball centred at a vertex of some cone $C(Y)$ with $Y\approx S^2$, then there is a simply connected domain $\mathbb{B}$ such that 
\begin{equation}
B_{1-\eps}(p)\subset \mathbb{B}\subset B_{1+\eps}(p)\, .
\end{equation}
\end{proposition}

The domain $\mathbb{B}$ in \autoref{prop:conS21c} is a sublevel set $\{b_p<r\}$. Here $b_p:=G_p^{-1}$ (up to normalization) and $G_p$ is a local Green's function $G_p$ with pole at $p$ such that the corresponding level set $\{b_p=r\}$ is a good slice as obtained through \autoref{prop:Greensphere}. We will call $\mathbb{B}_r(p):=\{b_p<r\}$ a good \textbf{Green-ball} and $\mathbb{S}_r(p):=\{b_p=r\}$ a good \textbf{Green-sphere} in this situation. 

\begin{remark}\label{rm:goodGBsimcon}
In the context of \autoref{prop:conS21c} we can actually prove that every good Green-ball $B_{1/2}(p)\subset\mathbb{B}_r(p)\subset B_1(p)$ is simply connected.
\end{remark}

\begin{remark}
We warn the reader that Green-balls and Green-spheres are not uniquely defined, as they depend on choosing a local Green's function $G_p$ (equivalently, a Green-distance $b_p$). This ambiguity does not cause any trouble.
\end{remark}

The relevance of \autoref{prop:conS21c} depends on the possibility of verifying its assumptions sufficiently often. This is possible thanks to the following lemma, originally due to Cheeger and Colding for smooth Riemannian manifolds, in \cite[Theorem 4.91]{CheegerColding96}, and generalized to $\RCD$ spaces thanks to De Philippis and Gigli's \cite{DePhilippisGigli}:
 
\begin{lemma}\label{lemma:conical}
Let $(X,\dist,\haus^3)$ be an $\RCD(-2,3)$ space such that $\haus^3(B_1(p))>v>0$ for every $p\in X$. For every $\eps>0$ there exist $C=C(\eps,v)>0$ and $\rho=\rho(\eps,v)>0$ such that the following holds. For every $p\in X$ and every $0<r<\rho$ there exists $r<r'<Cr$ such that 
\begin{equation}\label{eq:closeY}
\dist_{\mathrm{GH}}\left(B_{20r'}(p),B_{20r'}(o)\right)<20\eps r'\, ,
\end{equation}
where $B_{2r'}(o)\subset C(Y)$ is the ball centred at a vertex of a cone $C(Y)$, and $(Y,\dist_Y,\haus^2)$ is an $\RCD(1,2)$ space. 
\end{lemma}

\begin{proof}[Sketch of proof of \autoref{prop:1conn}]
The conclusion follows from \autoref{lemma:conical} and \autoref{prop:conS21c} provided that we can show that all the sections of the cones appearing in \eqref{eq:closeY} are homeomorphic to $S^2$. Note that such sections cannot be topological disks by the (no-)boundary stability \cite[Theorem 1.6]{BrueNaberSemolabdry}. Moreover, they are homeomorphic to $S^2$ for every sufficiently small $r$, depending on $p$, thanks to the assumption on the blow-ups. In the context of \autoref{thm:mainIntro}, it is possible to rule out sections $Y\approx \mathbb{RP}^2$ exploiting \autoref{thm:noRP2BPS}. For a general $\RCD(-2,3)$ space as in the assumptions of \autoref{thm:manrecRCD}, the argument is more delicate, and we omit it.
\end{proof}


\subsection{Good Green-spheres are nicely embedded}\label{sec:niceembd}

If $(X,\dist)$ is a smooth Riemannian manifold, we observed in \autoref{rm:noncrit} that the good Green-spheres $\mathbb{S}_r(p)$ are regular level sets of smooth functions. Hence, they are smoothly embedded submanifolds, bounding the respective Green-balls $\mathbb{B}_r(p)$. Recall, however, that there exist topological embeddings $\iota :S^2\to \mathbb{R}^3$ such that $\iota(S^2)$ separates $\R^3$ into two components, whose closures are not manifolds with boundary. A well-known example is Alexander's horned sphere. It is crucial that for (generic) good Green-spheres this kind of pathologies cannot occur.
\medskip

The set of good radii in \autoref{prop:Greensphere} has a positive measure.\\ 
\textbf{Assumption:} From now on, we will assume that a good Green-sphere is always a slice corresponding to a good value which is an accumulation point of good values. Almost every good radius obtained in \autoref{prop:Greensphere} satisfies such restriction, by Lebesgue's density theorem. 

\begin{definition}\label{def:kcoco}
Let $(X,\dist)$ be a metric space. A subset $C\subseteq X$ is said to be \emph{locally $1$-coconnected} (abbreviated to $1$-LCC) if every neighbourhood $U\subseteq X$ of an arbitrary point $x\in X$ contains another neighbourhood $V\subseteq X$ such that all continuous maps $\partial I^{2}\to V\setminus C$ extend to maps $I^{2}\to U\setminus C$, where $I:=[0,1]$.
\end{definition}


Exploiting the approximation from both sides (i.e., from the interior and from the exterior of the respective Green-ball) with locally uniformly contractible good slices, we can prove:

\begin{proposition}\label{prop:tame}
Any good Green-sphere $\mathbb{S}_r(p)$ is $1$-LCC. 
\end{proposition}

R.-H. Bing proved in \cite{Bing 1ULC} that $1$-LCC subsets of a $3$-manifold which are homeomorphic to closed surfaces are tamely embedded.

\begin{corollary}\label{cor:manifoldwithboundary}
If for a good Green-ball it holds $\overline{\mathbb{B}}_r(p)\subset X\setminus\mathcal{S}_{\rm{top}}$, then $\overline{\mathbb{B}}_r(p)$ is a $3$-manifold with boundary, with boundary $\mathbb{S}_r(p)$.
\end{corollary}

We can readily use \autoref{cor:manifoldwithboundary} to prove the implication from manifold regularity to the blow-ups being homeomorphic to $\R^3$ of \autoref{thm:manrecRCD}. Indeed, if $(X,\dist)$ is a topological $3$-manifold we can apply \autoref{cor:manifoldwithboundary} to any good Green-ball $\overline{\mathbb{B}}_r(p)$. If $r$ is sufficiently small, then $\mathbb{S}_r(p)$ is homeomorphic to the cross-section of any blow-up at $p$ by \autoref{prop:loccont3d} (ii). Since $\mathbb{S}_r(p)$ also bounds a $3$-manifold, $\mathbb{S}_r(p)\approx S^2$.   

\medskip

\begin{corollary}\label{cor:goodGman}
Any good Green-ball $\overline{\mathbb{B}}_r(p)\subset X\setminus\mathcal{S}_{\rm{top}}$ is homeomorphic to the closed Euclidean ball in $\R^3$.
\end{corollary}

\begin{proof}
Any such good Green-ball $\overline{\mathbb{B}}_r(p)$ is a simply connected $3$-manifold with boundary homeomorphic to $S^2$. The conclusion follows from the resolution of the Poincar\'e conjecture. 
\end{proof}

\begin{remark}
If one is only interested in the conclusion that $\overline{\mathbb{B}}_r(p)$ is contractible, then by Whitehead's theorem it is sufficient to argue that it has the homotopy type of a point. This can be verified with a standard argument based on Lefschetz duality; we address the reader to the proof of \cite[Proposition 9.23]{BruePigatiSemola} for the details. 
\end{remark}

\subsection{Local linear contractibility}\label{sec:loccont}

The proof of the local linear contractibility of $(X,\dist)$ is one of the most delicate steps of the argument. We begin with a precise statement:

\begin{proposition}\label{prop:loccontr}
Let $v>0$. There exist $C=C(v)>0$ and $\rho=\rho(v)>0$ such that if $(X,\dist,\haus^3)$ is an $\RCD(-2,3)$ space with $\haus^3(B_1(p))\ge v$ for any $p\in X$ and such that all blow-ups are homeomorphic to $\R^3$, then $B_r(p)$ is contractible inside $B_{Cr}(p)$ for every $r\le\rho$ and every $p\in X$.
\end{proposition}

As for local linear simple connectedness before, \autoref{prop:loccontr} follows from \autoref{lemma:conical} and the more precise statement that all good Green-balls are contractible, corresponding to \cite[Proposition 10.2]{BruePigatiSemola}. 
\medskip

The contractibility of the good Green-balls can be reduced to showing that they are $k$-connected for each $k\le 3$, i.e., they have trivial $k$-homotopy groups. See \cite[Proposition 10.7]{BruePigatiSemola} for the details about this reduction.
We prove this claim by arguing by induction over $k$. The base step of the induction, corresponding to $k=1$, follows from \autoref{prop:conS21c} (see also \autoref{rm:goodGBsimcon}).

\begin{proof}[Idea of the proof of the inductive step]
By Hurewicz, it is sufficient to argue that $H_{k+1}(\mathbb{B}_r(p);\mathbb{Z})$ is trivial for each good Green-ball $\mathbb{B}_r(p)$. We assume that $H_{k}(\mathbb{B}_s(q);\mathbb{Z})$ is trivial for each good Green-ball $\mathbb{B}_s(q)$. 

We consider a $(k+1)$-cycle $[\sigma]$ supported in $\mathbb{B}_r(p)$. By construction $\mathbb{B}_r(p)$ looks (scale invariantly) very close in the $\mathrm{GH}$ sense to $B_r(o)\subset C(Y)$, where $(Y,\dist_Y)$ is a topological sphere with $\mathrm{Sect}\ge 1$ in the Alexandrov sense. By the Cheeger-Colding Reifenberg theorem, we understand the topology of $\mathbb{B}_r(p)$ completely away from a (scale invariantly) small tubular neighbourhood $U\subset \mathbb{B}_r(p)$ of $C(\{p_1,\dots,p_{\ell}\})$, where $p_1,\dots,p_{\ell}\in Y$ are the points whose blow-up is not $\eps$-close to $\R^2$. In particular, we can deformation retract $\mathbb{B}_r(p)$ onto $U$. Hence $[\sigma]$ is homologous to a $(k+1)$-cycle $[\sigma']$ supported in $U$. 

With a careful covering argument and Mayer-Vietoris, we can break up $[\sigma']$ into a homologous sum $\sum[\sigma'_j]$ of $(k+1)$-cycles each supported in a good Green-ball $\mathbb{B}_s(q)$ with a diameter much smaller than the original $\mathbb{B}_r(p)$. Here the inductive hypothesis plays a key role. 

Morally, one would like to iterate this procedure with each of the cycles $[\sigma'_j]$ and pass to the limit until $[\sigma]$ vanishes in homology. See \cite[Lemma 10.9]{BruePigatiSemola} for the key technical lemma that justifies this limiting procedure.
Making this moral work requires quite some care. In particular, it is crucial to know that good Green-balls supported away from $\mathcal{S}_{\rm{top}}$ (which might be nonempty at this stage of the proof) are contractible, see \autoref{cor:goodGman}.  
\end{proof}


\begin{remark}
Inductive and iterative arguments with a similar flavour have appeared before in the study of the topology of spaces with lower Ricci bounds. For instance, in \cite{Perelmanmaximal}, or more recently, in \cite{PanWei,WangRCD}. With respect to these references, a key difference is the need to argue by contradiction with the help of Baire's category theorem, to gain some uniformity which might not be guaranteed a priori. See the discussion at the beginning of \cite[Section 10]{BruePigatiSemola} for the details of this reduction.
\end{remark}

\subsection{Generalized manifold regularity}\label{sec:genman}

To complete the proof that $(X,\dist)$ is a generalized $3$-manifold we are left with the verification of the isomorphism $H_{*}(X,X\setminus\{x\};\Z)\cong H_{*}(\R^3,\R^3\setminus\{0\};\Z)$ for all $x\in X$. 

For a $3$-manifold $M$, the sought isomorphism is easily established by considering a closed neighbourhood $\overline{U}_x\ni x$ homeomorphic to the closed ball $\overline{D}$ centred at the origin of $\R^3$. Indeed
\begin{equation}
H_{*}(M,M\setminus\{x\};\Z)\cong H_{*}(\overline{U}_x,\overline{U}_x\setminus\{x\};\Z)\, ,
\end{equation}
by excision, and clearly
\begin{equation}
H_{*}(\overline{U}_x,\overline{U}_x\setminus\{x\};\Z)\cong H_{*}(\R^3,\R^3\setminus\{0\};\Z)\, .
\end{equation}
We replace the neighbourhood homeomorphic to $\overline{D}$, whose existence is not clear at this stage, with a sufficiently small good (closed) Green-ball $\overline{\mathbb{B}}_r(x)$. The local uniform contractibility of $X$, i.e., \autoref{prop:loccontr}, together with the local uniform contractibility of the boundary Green-sphere $\overline{\mathbb{S}}_r(x)$, i.e., \autoref{prop:loccont3d} (i), yield:

\begin{lemma}{\cite[Lemma 10.12]{BruePigatiSemola}}\label{lemma:retract}
Any punctured good Green-ball $\overline{\mathbb{B}}_r(x)\setminus\{x\}$ deformation retracts onto the Green-sphere $\mathbb{S}_r(x)$, if $r$ is sufficiently small. 
\end{lemma}


Exploiting \autoref{lemma:retract}, the contractibility of $\overline{\mathbb{B}}_r(x)$, and the homeomorphism between $\mathbb{S}_r(x)$ and $S^2$, proving that 
\begin{equation}
H_{*}(\overline{\mathbb{B}}_r(x),\overline{\mathbb{B}}_r(x)\setminus\{x\};\Z)\cong H_{*}(\R^3,\R^3\setminus\{0\};\Z)\, 
\end{equation}
is an easy exercise in algebraic topology. 

\subsection{Manifold regularity}
The last step of our roadmap corresponds to upgrading the information that $(X,\dist)$ is a generalized $3$-manifold to the sought statement that it is a topological $3$-manifold.

\begin{remark}
In \cite[Section 4]{Semmes96}, S. Semmes constructed some Ahlfors $3$-regular generalized $3$-manifolds $(X,\dist)$, which are topological manifolds away from a single point which is not a manifold point. 
\end{remark}

\begin{remark}
The spherical suspension over the Poincar\'e homology sphere (endowed with a round metric with constant curvature $\equiv  1$) is a well-known example of a $4$-dimensional Alexandrov space with curvature $\ge 1$, which is a generalised $4$-manifold but is not a topological $4$-manifold. It is not a noncollapsed limit of manifolds with sectional curvature uniformly bounded from below by Perelman's stability theorem \cite{Perelman99}. On the other hand it is a noncollapsed Ricci limit space.
\end{remark}

The two remarks above are meant to convince the reader that the last step in our roadmap requires some work. We start again from two toy models.

\begin{proposition}\label{prop:toymodelglobal}
Let $(M^3,g)$ have $\Ric\ge 0$ and Euclidean volume growth. Then $M\approx \R^3$.
\end{proposition}

\begin{proof}[Sketch of the proof]
We fix $p\in M$ and consider the Green's function of the Laplacian $G_p:M^3\setminus\{p\}\to (0,\infty)$. There exists a sequence $t_i\to 0$ such that the super level sets $\{G_p\ge t_i\}$ are simply connected $3$-manifolds with boundary, with boundary homeomorphic to $S^2$. By the resolution of the Poincar\'e conjecture the $3$-manifold $\{G_p\ge t_i\}$ is homeomorphic (actually diffeomorphic) to the closed Euclidean $3$-ball. Hence $M^3$ admits an exhaustion into Euclidean $3$-balls. By \cite{Brown}, $M\approx \R^3$. 
\end{proof}

\begin{remark}
The statement of \autoref{prop:toymodelglobal} is certainly not original. Our goal above was to present a proof based on the methods developed in \cite{BruePigatiSemola}. There are (at least) two other morally independent proofs of \autoref{prop:toymodelglobal}:
\begin{itemize}
\item[i)] One can rely on the work of G. Liu \cite{Liu13} where complete (noncompact) $(M^3,g)$ with $\Ric\ge 0$ are classified up to diffeomorphism. The options which are not diffeomorphic to $\R^3$ are easily ruled out by the Euclidean volume growth condition.\\ 
The proof in \cite{Liu13} relies on the resolution of the Poincar\'e conjecture in the very last step to prove that $M^3$ is irreducible, i.e., any embedded $S^2$ bounds a standard $3$-ball. As pointed out to me by Chao Li, under the Euclidean volume growth condition one can check this in an alternative way by observing that $(M^3,g)$ admits an exhaustion into bounded strictly mean convex domains, see for instance \cite{FogagnoloMazzieri}. 
 Indeed, any strictly mean convex domain in a $3$-manifold with $\Ric\ge 0$ is diffeomorphic to a handlebody, as proved by N.-G. Anonov, Yu.-D. Burago, and V.-A. Zalgaller in \cite{AnonovBuragoZalgaller}, and independently by W.-H. Meeks III, L. Simon, and Yau in \cite{MeeksSimonYau} with a different method.
\item[ii)] Using the work of Simon and Topping \cite{SimonTopping22,SimonTopping22a} based on Ricci flow, it is possible to exhaust $M^3$ with open domains homeomorphic to the ball centred at the vertex of any blow-down cone of $(M^3,g)$. Such balls are homeomorphic to the Euclidean ball. The conclusion follows again from \cite{Brown}.
\end{itemize} 
Both approaches rely on the smoothness of $(M^3,g)$ very heavily. On the other hand, the proof of \autoref{prop:toymodelglobal} that we sketched above works for any $\RCD(0,3)$ manifold $(X,\dist,\haus^3)$ with Euclidean volume growth. 
\end{remark}

To prove that an $\RCD(-2,3)$ space $(X,\dist,\haus^3)$ with all blow-ups homeomorphic to $\R^3$ is a topological $3$-manifold, we need to contend with the presence of a possibly $1$-dimensional nonmanifold set $\mathcal{S}_{\rm top}$. It is instructive to discuss the case of isolated point singularities first.

\begin{proposition}\label{prop:toymodelisol}
Let $(X,\dist,\haus^3)$ be an $\RCD(-2,3)$ space such that $X\setminus\{x_1,\dots,x_k\}$ is a topological $3$-manifold, i.e., $\mathcal{S}_{\rm{top}}\subset \{x_1,\dots,x_k\}$, and every blow-up of $X$ at $x_i$ is homeomorphic to $\R^3$, for every $i=1,\dots, k$. Then $X$ is a topological $3$-manifold. 
\end{proposition}

\begin{proof}[Sketch of the proof]
The argument is local. We assume that there is only one singularity, i.e., $k=1$, and set $x:=x_1$. The key idea is to ``reverse'' the moral that we exploited in the proof of \autoref{prop:toymodelglobal}. We consider a strictly decreasing sequence $t_i\to 0$ such that the Green-balls $\mathbb{B}_{t_i}(x)$ are all good. Thanks to the assumption on the blow-ups at $x$, we can assume that all the corresponding Green-spheres $\mathbb{S}_{t_i}(x)$ are homeomorphic to $S^2$. Since $X\setminus \{x\}$ is a topological $3$-manifold, all the Green-annuli $\overline{\mathbb{B}}_{t_i}(x)\setminus \mathbb{B}_{t_{i+1}}(x)$ are $3$-manifolds with boundary (cf. with \autoref{prop:tame}), with two boundary components homeomorphic to $S^2$. Arguing as in the proof of \autoref{lemma:retract}, we infer that all such annuli are homotopically equivalent to $S^2$. By the resolution of the Poincar\'e conjecture, they are homeomorphic to $S^2\times[0,1]$. We can ``paste together'' these homeomorphisms into a homeomorphism between $\overline{\mathbb{B}}_{t_0}(x)\setminus\{x\}$ and $\overline{B}_1(0)\setminus\{0\}\subset \R^3$. The conclusion follows by the uniqueness of the one-point compactification. 
\end{proof}

The general case is clearly more challenging with respect to \autoref{prop:toymodelisol}. Indeed, the nonmanifold set $\mathcal{S}_{\rm{top}}$ in principle might be $1$-dimensional. Hence, reconstructing the topology of $X$ starting from the topology of $X\setminus\mathcal{S}_{\rm{top}}$ is highly nontrivial.
What is particularly important for us is that, although removing $X\setminus\mathcal{R}_{\eps}$ from $X$ might mess up the local simple connectedness, we can control the way this occurs. To make this precise, following \cite{Thickstunb}, we introduce:

\begin{definition}\label{def:GPDO}
If $X$ is a generalized $3$-manifold and $A\subset X$ is a closed subset, we say that $A$ has \emph{general-position dimension one} in $X$ if any continuous map $f:\overline D\to X$ can be approximated arbitrarily well by maps $g:\overline D\to X$ such that $g(\overline D)\cap A$ is $0$-dimensional. Here $\overline D\subset \R^2$ denotes the closed $2$-ball.
\end{definition}

\begin{proposition}{\cite[Proposition 11.1]{BruePigatiSemola}}\label{prop:GPDOS}
Let $(X,\dist,\haus^3)$ be an $\RCD(-2,3)$ space such that all the blow-ups are homeomorphic to $\R^3$. Then $X\setminus\mathcal{R}_{\eps}$ has general-position dimension one. A fortiori, $\mathcal{S}_{\rm{top}}\subset X\setminus\mathcal{R}_{\eps}$ has general-position dimension one. 
\end{proposition}

\begin{proof}[Sketch of the proof]
There are two main steps:
\begin{itemize}
\item[i)] for every good Green-sphere $\mathbb{S}_{r}(p)$, $\mathbb{S}_{r}(p)\cap (X\setminus\mathcal{R}_{\eps})$ is a finite set; 
\item[ii)] any continuous map $f:\overline D\to X$ can be approximated aribtrarily well by continuous $f_{\eps}:\overline D\to X$ such that 
\begin{equation}
B_{\eps}(\mathcal{S}_{\rm top})\cap f_{\eps}(\overline{D})\subset \bigcup_{i\le k}\mathbb{S}_{r_i}(p_i)\, ,
\end{equation} 
where $B_{\eps}(C)$ denotes the $\eps$-enlargment of a set $C\subset X$.
\end{itemize}
Step (i) exploits \autoref{prop:Greensphere} (iv). Roughly speaking, on $X\setminus\mathcal{R}_{\eps}$ there is at most one approximate splitting at every scale. By \autoref{prop:Greensphere} (iv), this splitting must be in the normal direction of the slice. This forces the Green-spheres to intersect $X\setminus\mathcal{R}_{\eps}$ transversely, and hence finiteness of the intersection set. 

For Step (ii), we can cover $X\setminus\mathcal{R}_{\eps}$ with a locally finite collection of good Green-balls $\mathbb{B}_{r_i}(p_i)$. After perturbing away $f(\overline{D})$ from the centers $p_i$, we can apply \autoref{lemma:retract} to push it iteratively onto the boundary Green-spheres $\mathbb{S}_{r_i}(p_i)$. 
\end{proof}

\begin{remark}
Step (i) above generalizes the well-known statement that for a $2$-dimensional Alexandrov space $(Y,\dist_Y)$ with empty boundary, the effective singular set $Y\setminus\mathcal{R}_{\eps}(Y)$ is locally finite. 
\end{remark}

Through the work of T.-L. Thickstun \cite{Thickstuna,Thickstunb} and Perelman's resolution of the Poincar\'e conjecture, \autoref{prop:GPDOS} implies that the generalized $3$-manifold $(X,\dist)$ is resolvable, i.e., there exist a $3$-manifold $N$ and a proper, cell-like, surjective, and continuous map $\Phi:N\to X$.

To complete the proof of \autoref{thm:manrecRCD}, we borrow a general recognition theorem from the work of R.-J. Daverman and D. Repovš \cite{DavermanRepovs}:

\begin{theorem}{\cite[Proposition 1.2, Theorem 3.4]{DavermanRepovs}}\label{thm:recoDavRep}
A resolvable generalized $3$-manifold $(X,\dist)$ is a $3$-manifold if any $x\in X$ is $1$-LCC and admits arbitrarily small neighbourhoods $U$ such that there exist maps $g:S^2\to U\setminus\{x\}$ with the following properties:
\begin{itemize}
\item[(i)] $g:S^2\to g(S^2)\subset X$ is a homeomorphism;
\item[(ii)] $g(S^2)$ is $1$-LCC in $X$;
\item[(iii)] $g:S^2\to U$ is homotopically trivial;
\item[(iv)] $g:S^2\to U\setminus\{x\}$ is not homotopically trivial.
\end{itemize}
\end{theorem}

\begin{proof}[Completing the proof of \autoref{thm:manrecRCD}]
We are left with the proof of the implication from the blow-ups being homeomorphic to $\R^3$ to $X$ being a topological $3$-manifold. We argued above that $(X,\dist)$ is a resolvable generalized $3$-manifold. The maps $g$ as in the assumptions of \autoref{thm:recoDavRep} can be taken to be homeomorphic parameterizations of (sufficiently small) good Green-spheres $\mathbb{S}_r(x)$. That such good Green-spheres are homeomorphic to $S^2$ follows from the assumption about the blow-ups of $(X,\dist)$ and \autoref{prop:loccont3d}. Item (ii) corresponds to \autoref{prop:tame}. Item (iii) follows from \autoref{prop:loccontr}. Item (iv) follows from \autoref{lemma:retract}. 
\end{proof}

\subsection{Remarks}
We conclude this section with an informal discussion about some of the key $3$-dimensional aspects of the proof of \autoref{thm:manrecRCD}.
\begin{itemize}
\item[(i)] The almost conicality away from uniformly finitely many scales at all points, i.e., \autoref{lemma:conical}, holds in any dimension. However, being close to a cone in the Gromov-Hausdorff sense is enough to activate a slicing mechanism only in dimension $3$, unless the cone almost splits a factor $\R^{n-3}$, see \autoref{rm:slicingfails}. 
\item[(ii)] In any dimension $n$, the combination of a lower bound on the Ricci curvature with a lower bound on the volume yields uniform finiteness for fundamental groups. However, while for a (closed) $3$-manifold the fundamental group carries a great deal of information about the topology, this is not the case in higher dimensions. Indeed, in the counterexamples to a manifold recognition theorem based on the blow-up behaviour discussed in \autoref{rm:nomanrec4}, the main issues are at the level of the second homotopy groups.
\item[(iii)] As clearly illustrated by \autoref{thm:3d}, in dimension $3$ a qualitative information, i.e., being a manifold, combined with the nonnegativity of Ricci curvature and Euclidean volume growth, uniquely determines the large-scale topological behaviour, i.e., the blow-down is homeomorphic to $\R^3$. The proof of \autoref{thm:manrecRCD} heavily hinges on this moral. On the regions where the manifold regularity is already known, we propagate this qualitative information from infinitesimal scales to finite ones, as in \autoref{cor:goodGman}. Simultaneously, we reverse the moral to push the \emph{nonmanifoldness} to infinitesimal scales, where we can exploit the information on the blow-up behaviour, as in the proof of \autoref{prop:loccontr}.
\end{itemize}

\subsection{The stability \autoref{thm:stability3d}}

The proof of \autoref{thm:stability3d} hinges on a local uniform contractibility statement:

\begin{theorem}{\cite[Theorem 1.10]{BruePigatiSemola}}\label{thm:locunicontr}
Let $v>0$ be fixed. There exist constants $C=C(v)>0$ and $\rho=\rho(v)>0$ such that if $(X,\dist,\haus^3)$ is an $\RCD(-2,3)$ topological manifold with $\haus^3(B_1(p))\ge v$ for any $p\in X$, then the ball $B_r(p)$ is contractible inside $B_{Cr}(p)$ for every $0<r\le\rho$ and every $p\in X$.
\end{theorem}

\begin{remark}
For smooth Riemannian manifolds \autoref{thm:locunicontr} corresponds to \cite[Proposition 3.1]{Zhu93}, whose proof is based on a different idea. It is conceivable that the arguments in \cite{Zhu93} could be used to prove \autoref{thm:locunicontr} as well. 
\end{remark}

\begin{remark}
Note that \autoref{thm:locunicontr} fails to extend to higher dimensions. Indeed, the local uniform $1$-contractibility of the class of $n$-dimensional Riemannian manifolds with Ricci curvature and volume uniformly bounded from below fails as soon as $n\ge 4$; see \cite[Remark 2, pg. 262]{Otsu}.
\end{remark}  

We already proved \autoref{thm:locunicontr} in the previous section when we proved \autoref{prop:loccontr}. 

\begin{remark}\label{rm:Semmes*}
Actually, under the same assumptions, we proved something stronger: for every $p\in X$ and every $0<r<\rho$, there is an open domain $B_r(p)\subset U\subset B_{Cr}(p)$ homeomorphic to the Euclidean $3$-ball. 
\end{remark}

We can exploit \autoref{thm:locunicontr} in combination with the results from \cite{Petersen90} to establish a weak version of \autoref{thm:stability3d}, with homeomorphisms replaced by homotopy equivalences. 

\begin{proposition}\label{prop:epsequiv}
Under the assumptions of \autoref{thm:stability3d} there exists $i_0\in \N$ such that for every $i\ge i_0$ there exist $\epsilon_i$-equivalences $f_i:X_i\to X$ with $\epsilon_i\to 0$ as $i\to \infty$, i.e., the $f_i$'s are continuous and there exists continuous $g_i:X\to X_i$ such that, for every $i\ge i_0$: 
\begin{itemize}
\item[i)] $f_i\circ g_i$ is homotopic to $\mathrm{id}_X$ through a homotopy $G_i$;
\item[ii)] $g_i\circ f_i$ is homotopic to $\mathrm{id}_{X_i}$ through a homotopy $F_i$;
\item[iii)] all the flow lines of $F_i$ and $G_i$ have diameter less than $\epsilon_i$.
\end{itemize}
\end{proposition}

The imprecise idea is that local uniform contractibility removes the obstruction to extend maps continuously.
\medskip

While the proof of \autoref{thm:locunicontr} in \cite{BruePigatiSemola} is (essentially) self-contained, the proof of \autoref{thm:stability3d} is definitely not, as it heavily relies on some deep tools from controlled homotopy theory. We sketch it below.

\begin{proof}[Sketch of the proof of \autoref{thm:stability3d}]
All the blow-ups of the limit $(X,\dist)$ in \eqref{eq:limstab} have cross-section homeomorphic to $S^2$ by a mild generalisation of \autoref{thm:noRP2BPS}. Hence, $X$ is a topological $3$-manifold, by \autoref{thm:manrecRCD}. Thus, we can invoke the $3$-dimensional $\alpha$-approximation theorem due to W. Jakobsche in \cite{Jakobsche}, taking into account the resolution of the Poincaré conjecture, to perturb the $\epsilon_i$-equivalences in \autoref{prop:epsequiv} to homeomorphisms. 
\end{proof}

\begin{remark}
The terminology \emph{$\alpha$-approximation} comes from the work of T.-A. Chapman and S. Ferry \cite{ChapmannFerry}, where an analogous statement was obtained before in dimensions $n\ge 5$.
\end{remark}

\begin{remark}
The proof of \autoref{thm:stability3d} is very similar in spirit to the proof of Perelman's stability theorem for noncollapsing limits of smooth Riemannian manifolds with sectional curvature bounded below discussed in \cite[Section 3]{KapovitchPer}. 
\end{remark}

\section{Open questions}\label{sec:OpenQuest}

We end this survey with a collection of conjectures and open questions. Some of them are well-known, and some are original, to the best of the author's knowledge.  

\subsection{Improving the regularity in the manifold recognition }

In the context of the manifold recognition \autoref{thm:manrecRCD}, it would be interesting to understand whether a more regular homeomorphism with a smooth Riemannian manifold can be constructed. In \cite{BruePigatiSemola} we raised the following:

\begin{conjecture}{\cite[Conjecture 1.18]{BruePigatiSemola}}\label{conj:biHolder}
Let $(X,\dist,\haus^3)$ be an $\RCD(-2,3)$ space such that all blow-ups are homeomorphic to $\R^3$. Then $(X,\dist)$ is locally biH\"older homeomorphic to a smooth, complete Riemannian manifold $(M^3,g)$. 
\end{conjecture}

If $(X^3,\dist)$ is a noncollapsed Ricci limit, then \autoref{conj:biHolder} holds, thanks to the already mentioned works of Simon and Simon and Topping \cite{Simon14,SimonTopping22,SimonTopping22a}, and A. McLeod and Topping \cite{McLeodTopping}.\\ 
As we stressed in \autoref{rm:Semmes*}, an $\RCD(-2,3)$ space $(X,\dist,\haus^3)$ with all blow-ups homeomorphic to $\R^3$ is a manifold in an ``effective'' way, at least locally. This corresponds to condition $(\ast)$ in Seemes' \cite[Definition 1.3]{Semmes96}. As shown therein, there exist Ahlfors $3$-regular metric $3$-manifolds satisfying condition $(\ast)$ which are not biH\"older homeomorphic to smooth Riemannian $3$-manifolds. It seems conceivable that $\RCD(-2,3)$ manifolds satisfy condition $(\ast\ast)$ in \cite[Definition 1.7]{Semmes96} as well. To the best of the author's knowledge, it is an open question whether the combination of Ahlfors regularity and the said condition $(\ast\ast)$ yields the existence of locally biH\"older homeomorphisms with a smooth Riemannian manifold for general metric manifolds; see the discussion in \cite[Section 12]{Semmes96}.\\
Of course, establishing \autoref{conj:biHolder} would pave the way for a biH\"older version of the stability \autoref{thm:stability3d}.
\medskip

For $3$-dimensional Alexandrov spaces $(X^3,\dist)$ which are topological manifolds, the biLipschitz version of the stability theorem (an unpublished result of Perelman, see the discussion in \cite{KapovitchPer}) together with the argument sketched at the very end of Section \ref{sec:prelle3} yield (locally) biLipschitz homeomorphisms with subsets of $\R^3$ near to every point. It is natural to ask:

\begin{question}\label{q:biLip}
Is there any $\RCD(-2,3)$ space $(X,\dist,\haus^3)$ with all blow-ups homeomorphic to $\R^3$ which is not (locally) biLipschity homeomorphic to a smooth, complete, Riemannian $(M^3,g)$?
\end{question}

We stress that for a general $\RCD(-2,3)$ space $(X,\dist,\haus^3)$ as above, it is an open question whether there exists an open subset $U\subset X$ biLipschitz homeomorphic to an open set in $\R^3$. The question is open for noncollapsed Ricci limits as well, see \cite[pg. 411]{CheegerColding97I} and \cite[Open Problem 2.2]{Naberconj}. 

\subsection{Regularization}

The next question has been around for some time, although it did not appear in print before \cite{BruePigatiSemola}. See also the recent survey \cite[Section 7]{Simon24}.

\begin{question}\label{quest:smoothingRCD}
Let $(X,\dist,\haus^3)$ be an $\RCD(-2,3)$ space which is a topological manifold. Is it a noncollapsed Ricci limit space?
\end{question}

The analogous question for $3$-dimensional Alexandrov spaces with curvature bounded from below has been open from the Eighties; see, for instance, the discussion around \cite[2.4 Open Problem]{Lebedevaetal}. 
\medskip

We mentioned in Section \ref{sec:curvge0} that the cross-section of the blow-down of every $(M^n,g)$ with $\mathrm{Sect}\ge 0$ and Euclidean volume growth is a smoothable Alexandrov space homeomorphic to $S^{n-1}$, thanks to \cite{Kapovitchcross}. By \autoref{thm:3d}, this is the case also for every $(M^3,g)$ with $\Ric\ge 0$ and Euclidean volume growth. In view of \autoref{thm:mainIntro}, the next question can be seen as dealing with a special case of \autoref{quest:smoothingRCD}.

\begin{question}\label{quest:smoothingcross}
Let $(M^4,g)$ be smooth, complete, with $\Ric\ge 0$ and Euclidean volume growth. Let $(Z,\dist_Z,\haus^3)$ be the cross-section of a blow-down of $(M^4,g)$. Is $(Z,\dist_Z)$ a noncollapsed Ricci limit space?
\end{question}

Similarly, one might ask:

\begin{question}\label{quest:smoothingslice}
Let $(M^n_i,g_i,p_i)$ be smooth, complete Riemannian manifolds with $\Ric_i\ge-\delta_i$, $\delta_i\to 0$ as $i\to\infty$. Assume that 
\begin{equation}
(M^n_i,g_i,p_i)\xrightarrow{\mathrm{pGH}}(\R^{n-3}\times Z,\dist_{\R^{n-3}}\times\dist_Z,(0,p))\, ,\quad\text{as $i\to\infty$}\, , 
\end{equation}
and $\haus^3(B_r(p))\ge vr^3$ for every $r>0$, for some $v>0$. Is $(Z,\dist_Z)$ a noncollapsed Ricci limit space?
\end{question}

The slice $Z$ arising in \autoref{quest:smoothingslice} is homeomorphic to $\R^3$, by \cite[Theorem 1.7, Theorem 1.9]{BruePigatiSemola}. It seems tempting to approach \autoref{quest:smoothingslice} and \autoref{quest:smoothingcross} via Ricci flow. Thanks to an unpublished result due to R. Hochard \cite[Lemma I.3.12]{Hochard}, to answer \autoref{quest:smoothingslice} in the affirmative, it would suffice to construct a smooth complete Ricci flow $(M^n,g(t))_{t\in (0,T)}$ such that, for every $t\in(0,T)$ and for some $K>0$,
\begin{equation}
\left|\mathrm{Riem}(g(t))\right|\le \frac{K}{t}\, ,\quad  \mathrm{inj}_x(g(t))\ge \sqrt{\frac{K}{t}}\, ,\quad \Ric(g(t))\ge 0\, ,
\end{equation}
and
\begin{equation}\label{eq:convHochard}
\lim_{t\to 0}(M,\dist_{t})=(\R^{n-3}\times Z,\dist_{\R^{n-3}}\times\dist_Z)\, .
\end{equation}
We address the reader to Hochard's PhD thesis \cite{Hochard} for the precise notion of convergence employed in \eqref{eq:convHochard} and a thorough discussion about some of the difficulties related to this approach.

\subsection{Topology of $(M^4,g)$ with $\Ric\ge 0$ and Euclidean volume growth}

The examples with infinite topological type constructed by Menguy in \cite{Menguy inftop} clearly illustrate that the topology of manifolds within this class might be quite complicated. As far as the author is aware, the only general topological results are: 
\begin{itemize}
\item[i)] $|\pi_1(M)|<\infty$, due to Li \cite{Lilarge} and Anderson \cite{Andersonpi1} independently; 
\item[ii)] $H_3(M;\mathbb{Z})=0$, due to Y. Itokawa and R. Kobayashi in \cite{ItokawaKobayashi};
\item[iii)] $H_2(M;\mathbb{Z})$ is torsion-free, due to Z. Shen and C. Sormani in \cite{ShenSormani}.
\end{itemize}
Up to changing $3$ into $(n-1)$ and $2$ with $(n-2)$, (i), (ii), and (iii) above hold for any $(M^n,g)$ with $\Ric\ge 0$ and Euclidean volume growth. \\
Brena, Bru\`e and Pigati recently proved that any $(M^4,g)$ with $\Ric\ge 0$ and Euclidean volume growth is orientable in \cite{BrenaBruePigati}, partly relying on \autoref{thm:mainIntro}. Such a result does not generalize to $n>4$. Indeed, $\mathbb{RP}^2\times\mathbb{R}^3$ admits a complete metric $g$ with $\mathrm{Ric}\ge 0$ and Euclidean volume growth; see \cite{Otsu}.
It seems conceivable to the author that \autoref{thm:mainIntro} could help classify all possible fundamental groups of $(M^4,g)$ with $\Ric\ge 0$ and Euclidean volume growth.
\medskip

If $\Ric\equiv 0$, then $M^4$ can be compactified to a smooth manifold with boundary with boundary diffeomorphic to the cross-section of the blow-down. No such statement can hold, in general, when we only assume that $\Ric\ge 0$. 
It seems natural to ask whether the finiteness of the topological type is the only obstruction.

\begin{question}\label{quest:compactification}
Let $(M^4,g)$ be smooth complete with $\Ric\ge 0$, Euclidean volume growth, and finite topological type. Does there exist a compact topological manifold with boundary $\overline{M}$ such that $\partial M$ is homeomorphic to the cross-section of a blow-down of $(M,g)$ and $\overline{M}\setminus\partial M$ is homeomorphic to $M$?
\end{question}

In all the examples available in the literature at the time of writing, a nontrivial fundamental group for the cross-section of the blow-down of $(M^4,g)$ can only appear at the price of nontrivial homotopy groups for $M$.

\begin{question}\label{quest:contractcross}
Let $(M^4,g)$ be smooth complete with $\Ric\ge 0$ and Euclidean volume growth. Assume that $M$ is contractible. Can the cross-sections of the blow-downs of $(M,g)$ be not simply connected?
\end{question} 

The examples constructed by S. Zhou in \cite{Zhou23} suggest that \autoref{quest:contractcross} might be a subtle one.
\medskip

To the best of the author's knowledge, the unique smooth contractible $4$-manifold which is known to admit a complete metric with $\Ric\ge 0$ and Euclidean volume growth is $\R^4$ (with the standard smooth structure).

\begin{question}\label{quest:contrtopdif}
Do there exist contractible $4$-manifolds $M$ which admit a complete smooth metric $g$ with $\Ric\ge 0$ and Euclidean volume growth and are not homeomorphic (or diffeomorphic) to $\R^4$?
\end{question}

Note that if a contractible $M^4$ admits a Ricci-flat metric with Euclidean volume growth, then $M^4$ is diffeomorphic to $\R^4$ and the metric is flat. This follows from \cite[Lemma 6.3]{Anderson89} combined with \cite[Corollary 8.85]{CheegerNaber15}, stated here as \autoref{thm:ChNaALE}. 

As pointed out to me by Shengxuan Zhou, the questions discussed in this section are mostly open also in the case of Kähler surfaces, unless one assumes that they are polarized or Stein.

\subsection{Generalizations to $\RCD$ spaces}

It is conceivable that for every $\RCD(-2,3)$ space $(X,\dist,\haus^3)$ every $x\in X$ should have a neighbourhood homeomorphic to a blow-up of $(X,\dist)$ at $x$. Such a statement would (almost) extend Perelman's conical neighbourhood theorem for Alexandrov spaces with curvature bounded from below from \cite{Perelman99} to the present setting. Moreover, it would resolve in the affirmative the following:

\begin{conjecture}[Mondino '22]\label{conj:Mondino}
Every $\RCD(-2,3)$ space $(X,\dist,\haus^3)$ is homeomorphic to an orbifold (possibly with boundary). 
\end{conjecture}

To address \autoref{conj:Mondino}, one needs to handle blow-ups homeomorphic to $C(\mathbb{RP}^2)$, corresponding to interior orbifold points, or $\mathbb{R}^3_{+}$ (i.e., a half-space), corresponding to boundary points. The case of boundary points seems more delicate. In particular, to deal with it, it might be helpful to have positive answers to \cite[Question 4.9]{KapovitchMondino} and \cite[Open Question 7.3]{BrueNaberSemolabdry}.
\medskip

If \autoref{conj:Mondino} holds true, then it is very likely that the manifold assumption could be removed from the stability \autoref{thm:stability3d}. This would amount to settle in the affirmative the following:

\begin{conjecture}\label{conj:stability}
Let $(X_i,\dist_i,\haus^3)$ be $\RCD(-2,3)$ spaces. Assume that 
\begin{equation}\label{eq:limstab}
(X_i,\dist_i)\xrightarrow{\mathrm{GH}}(X,\dist)\, \quad \text{as $i\to\infty$}\, ,
\end{equation}
without collapse, for some compact $\RCD(-2,3)$ space $(X,\dist,\haus^3)$. Then there exists $i_0\in\N$ such that $X_i$ is homeomorphic to $X$ for every $i\ge i_0$.
\end{conjecture}

In the context of noncollapsed Gromov-Hausdorff convergence under lower sectional curvature bounds, cross-sections of blow-ups are topologically stable, as illustrated by Kapovitch's \cite{Kapovitchcross}. See, in particular, Theorem 5.1 and Remark 5.4 therein. The combination of \cite[Theorem 1.6]{BrueNaberSemolabdry} with \cite{BruePigatiSemola} shows that something similar holds for (synthetic) lower Ricci bounds in dimension $3$.

\begin{theorem}\label{thm:secstab}
Let $(X_i,\dist_i,\haus^3,p_i)$ be $\RCD(-2,3)$ spaces with empty boundary such that 
\begin{equation}
(X_i,\dist_i,\haus^3,p_i)\xrightarrow{\mathrm{pmGH}}(X,\dist,\haus^3,p)\, ,\quad \text{as $i\to\infty$}\, .
\end{equation}
Let $C(Y)$ be a blow-up of $(X,\dist)$ at $p$. There exist $q_i\in X_i$ such that the cross-sections of every blow-up of $(X_i,\dist_i)$ at $q_i$ are homeomorphic to $Y$.
\end{theorem} 
The examples discussed in Section \ref{subsec:alltop} clearly illustrate that the topology of cross-sections (of blow-ups) need not be stable under noncollapsed Gromov-Hausdorff convergence with lower Ricci bounds in dimensions larger than $4$. However, the singularities of the form $\R^{n-3}\times C(\mathbb{RP}^2)$ enjoy some form of stability, as shown by \autoref{thm:noRP2BPS}. We conjecture that \autoref{thm:noRP2BPS} generalizes to $\RCD$ spaces if suitably reformulated.

\begin{conjecture}\label{conj:noCRP2}
 Fix $n\ge 4$. Let $(X_i,\dist_i,\haus^n,p_i)$ be $\RCD(-\delta_i,n)$ spaces, with $\delta_i\to 0$, $\haus^n(B_1(p_i))>v>0$ and such that 
\begin{equation}\label{eq:pGH3sim}
(X_i,\dist_{i},p_i)\xrightarrow{\mathrm{pGH}}\R^{n-3}\times C(Y)\, ,\quad i\to \infty\, ,
\end{equation}
for some metric space $(Y,\dist_Y)$. Assume that each $(X_i,\dist_i)$ has no blow-up of the form $\R^{n-3}\times C(W)$ with $W\approx\mathbb{RP}^2$. Then $Y\approx S^2$ or $Y\approx \overline{D}^2$. If the $X_i$'s have empty boundaries, then only the first possibility can occur. 
\end{conjecture}

Besides the technical challenge in generalizing the slicing \autoref{thm:slicingBPS} to $\RCD$ spaces, the present proof of \autoref{thm:noRP2BPS} hinges on the smoothness in a crucial way. As such, proving \autoref{conj:noCRP2} would require some new idea. 
\medskip

We pointed out before that a (smooth) $3$-manifold $M^3$ admits a smooth Riemannian metric $g$ with $\Ric\ge 2$ if and only if it admits a distance $\dist$ compatible with the topology and such that $(M^3,\dist,\haus^3)$ is an $\RCD(2,3)$ space. We end the survey with the following:

\begin{conjecture}
There exist $n\ge 4$ and a smooth simply connected $n$-manifold $M^n$ such that:
\begin{itemize}
\item[i)] there is no smooth metric $g$ with $\Ric\ge n-1$ on $M^n$;
\item[ii)] there is a distance $\dist$ compatible with the topology of $M^n$ such that $(M^n,\dist,\haus^n)$ is an $\RCD(n-1,n)$ space.
\end{itemize}
\end{conjecture}

\end{document}